\documentclass[12pt]{amsart}
\usepackage{amsfonts}
\usepackage[backend=biber,style=alphabetic,sorting=nyt]{biblatex}
\usepackage{tikz-cd}
\usepackage{amsthm}
\usepackage{amsmath}
\usepackage{amscd}
\usepackage{amssymb}
\usepackage{stmaryrd}
\usepackage{fullpage}
\usepackage[colorlinks=false]{hyperref}
\newcommand*{\triple}[2][.1ex]{%
  \mathrel{\vcenter{\offinterlineskip%
  \hbox{$#2$}\vskip#1\hbox{$#2$}\vskip#1\hbox{$#2$}}}}
\newcommand*{\triplerightarrows}{\triple{\rightarrow}}

\newcommand{\ext}{\text{Ext}}
\newcommand{\coh}{\text{Coh}}
\newcommand{\spec}{\text{Spec}\,}
\newcommand{\Hom}{\text{Hom}}
\newcommand{\Pic}{\mathrm{Pic}}
\newcommand{\Picst}{\mathrm{Pic}^{\mathrm{st}}}
\newcommand{\AandA}{{\mathbb{A}^1\smash{\underset{*}{\sqcup}}\mathbb{A}^1}}
\newcommand{\MF}{\mathrm{MF}}

\theoremstyle{definition}
\newtheorem{thm}{Theorem}[section]
\newtheorem{lemma}[thm]{Lemma}
\newtheorem{prop}[thm]{Proposition}
\newtheorem{cor}[thm]{Corollary}

\newtheorem{rem}[thm]{Remark}
\newtheorem{defi}[thm]{Definition}
\newtheorem{example}[thm]{Example}

\title[Mirror Symmetry for Toric Calabi-Yau Three-Orbifolds]{Topological Fukaya Category and Mirror Symmetry for Toric Calabi-Yau Three-Orbifolds}

\begin{document}

\author{Qingyuan Bai}
\address{Qingyuan Bai, School of Mathematical Sciences, Peking University, 5 Yiheyuan Road, Beijing 100871, China}
\email{baiqingyuan1@gmail.com}

\author{Bohan Fang}
\address{Bohan Fang, Beijing International Center for Mathematical
  Research, Peking University, 5 Yiheyuan Road, Beijing 100871, China}
\email{bohanfang@gmail.com}

\begin{abstract}
	We prove a version of homological mirror symmetry statement for toric Calabi-Yau $3$-orbifolds, thus extending \cite{Pascaleff_2019} to the case of orbifolds under the mirror symmetry setting considered in \cite{fang2019remodeling}. The B-model is the matrix factorization category for the toric Calabi-Yau $3$-orbifold with a superpotential; while the A-model is a topologically defined Fukaya-type category on its mirror curve.
\end{abstract}

\maketitle

\section{Introduction}

The Fukaya category of a symplectic manifold is a categorification of its Lagrangian Floer theory (see \cite{FOOO, SeidelBook} for comphrensive introduction). It is an important object in symplectic geometry, and plays a central role in homological mirror symmetry (HMS) \cite{KontsevichICM}. Computing such Fukaya categories is a difficult problem in general. However, according to a proposal of Kontsevich \cite{KontsevichProposal}, the partially wrapped Fukaya category of an exact symplectic manifold should behave like a cosheaf supported on its ``skeleton'' -- its symplectic core, usually a singular Lagrangian submanifold. There have been numerous attempts to realize Kontsevich's proposal (see for example \cite{sibilla2011ribbon,dyckerhoff2013triangulated,Haiden_2017,nadler2016wrapped}, and more recently, \cite{Ganatra_2019,ganatra2019sectorial,ganatra2020microlocal}). A pleasant consequence of these construction is that, to compute the global Fukaya cateogry, one can first compute the relatively easy local pieces, then glue them together. This is especially useful in the setting of mirror symmetry. Given a mirror pair $(M,M^\vee)$, the HMS statement says $$\text{Fuk}(M)\cong \text{Coh}(M^\vee)$$ where the left side is the Fukaya category and the right side is the category of coherent sheaves, both understood as higher categories ($A_\infty$ or dg). There are various descent theorems in algebraic geometry allowing us to compute $\text{Coh}(M^\vee)$ in a local-to-global manner. The cosheaf property of Fukaya category is a descent phenomenon on the mirror side. Hence proving the homological mirror symmetry statement reduces to matching the local pieces and gluing data. This kind of consideration has already appeared in \cite{Kuwagaki_2020,lee2016homological,Pascaleff_2019}. 


Prompted by Kontsevich's proposal \cite{KontsevichProposal}, one may bypass the actual Floer-theoretic Fukaya categories, and take the topological nature of such categories for granted. One simply \emph{defines} a substitute of the Fukaya category by the topological data on the skeleton. A drawback of this approach is that the dependence of such categories on the choice of the symplectic core, which is not unique. The independence of the choice of skeleta could be worked out directly \cite{Dyckerhoff_2017, nadler2021sheaf} in many situations, or by identifying such definitions to the actual Fukaya categories which are invariants of symplectic manifolds \cite{ganatra2020microlocal}.

In this paper, we consider HMS for toric Calabi-Yau $3$-orbifolds using the topological definition for Fukaya categories. The enumerative mirror symmetry in this setting (matching Gromov-Witten invariants and variations of Hodge structures) dates back to \cite{AganagicVafa,AKV}. There are many works along this line in the mathematical and physics literature (e.g. \cite{BKMP,EOProof,fang2019remodeling}). The mirror to a toric Calabi-Yau $3$-orbifold is an affine curve called the \emph{mirror curve}. We consider HMS instead of enumerative mirror symmetry in this paper, and match the A-model on the mirror curve and the B-model on the toric Calabi-Yau $3$-orbifold. The homological mirror symmetry in this situation has been studied in \cite{abouzaid2014homological} for punctured spheres as mirror curves and for general smooth toric Calabi-Yau $3$-folds in \cite{lee2016homological} where the A-model is the wrapped Fukaya category. The works of \cite{nadler2016wrapped,Pascaleff_2019} use the topological approach for Fukaya categories in this setting, and in particular Pascaleff-Sibilla \cite{Pascaleff_2019} proves the general HMS statement for smooth toric Calabi-Yau $3$-folds. We extend the work of \cite{Pascaleff_2019} to the case of toric Calabi-Yau $3$-orbifolds under the mirror symmetry set-up in \cite{fang2019remodeling}.

\subsection{Main result}

Let $\mathcal X_\Sigma$ be a toric Calabi-Yau $3$-orbifold defined from the fan data $\Sigma$. Its mirror geometry $\mathcal{C}_\Sigma$ is a non-compact Riemann surface, called the mirror curve. The B-model of the toric orbifold is captured by the category of matrix factorization of the Landau-Ginzburg model $(\mathcal{X}_\Sigma,W_\Sigma)$. Here the superpotential $W_\Sigma:\mathcal{X}_\Sigma\rightarrow \mathbb{A}^1$ is given by the height function on the fan $\Sigma$. On the mirror side, we use the topological Fukaya category of \cite{dyckerhoff2013triangulated} to describe the symplectic geometry of the mirror curve $\mathcal{C}_\Sigma$. The main result of this note is the following.

\begin{thm} Denote $\text{MF}^\infty(\mathcal{X}_\Sigma,W_\Sigma)$ the Ind completion of the category of matrix factorization and $\mathcal{F}^{top}_\infty(\mathcal{C}_\Sigma)$ the Ind completion of the topological Fukaya category, then we have an equivalence of $\mathbb{Z}/2$-dg categories
$$\text{MF}^\infty(\mathcal{X}_\Sigma,W_\Sigma)\cong\mathcal{F}^{top}_\infty(\mathcal{C}_\Sigma).$$
\end{thm}

As a corollary, we get a crepant transformation type theorem. In general, an affine toric Calabi-Yau variety $X$ of dimension $3$ is given by the combinatorial data of a planar polygon $P$ with not-necessarily orbifold singularity. Such toric variety is associated to the cone over $P$ placed at height $1$ plane. There are many ways to resolve singularity: a triangulation of $P$ gives rise to a fan $\Sigma$ and it defines a smooth toric Calabi-Yau $3$-orbifold \cite{Borisov_2004}. When such a triangulation is unimodular, we obtain a smooth toric Calabi-Yau $3$-fold. Let $\Sigma$ and $\Sigma'$ be two such fans from two triangulations, and then $({\mathcal X}_{\Sigma'},\mathcal{X}_{\Sigma''})$ can be viewed as a crepant pair -- their canonical classes are the same as the pullback of the canonical class of the affine variety $X$.
\begin{cor} For the pair $(\mathcal{X}_{\Sigma'},{\mathcal X}_{\Sigma''})$,
$$\text{MF}^\infty(\mathcal{X}_{\Sigma'},W_{\Sigma'})\cong\text{MF}^\infty({\mathcal X}_{\Sigma''},W_{\Sigma''})$$
\end{cor}
\begin{proof}
Both sides are identified with the topological Fukaya category of  mirror curves. Now it suffices to note that they have the same mirror curve topologically.
\end{proof}

\subsection{Proof strategy} The proof of the result is similar to that of \cite{Pascaleff_2019}, with careful consideration of the mirror curve for an affine toric Calabi-Yau $3$-orbifold. As hinted above, we reduce the mirror equivalence to identifying local pieces and gluing data. First we consider the case when toric orbifold is affine. In this situation, we give a proof similar to that in \cite{nadler2016wrapped} which deals with the case $(\mathbb{A}^3,W=z_0z_1z_2)$ . On the toric geometry side, a direct computation of the equivariant matrix factorization category identifies it with the equivariant coherent sheaf category of a stacky curve. On the mirror curve side, we exhibit the mirror curve as an unramified cover of 3-punctured sphere (or the so-called `pair of pants'). One observes that the skeleton of 3-punctured sphere can be lifted to be a skeleton of the mirror curve, and the decomposition can be lifted as well (see Figure \ref{skeleton}). This gives a way to compute topological Fukaya category. Eventually both sides have equivariant structures which we take advantage of for comparison. We remark that such equivariant computation already appeared in \cite{abouzaid2014homological}. 

Then we make use of the main technical result of \cite{Pascaleff_2019} which gives an `exceptional gluing' of topological Fukaya categories, which corresponds to Zariski descent on the toric geometry side. For a general toric Calabi-Yau three orbifold, it is glued from local pieces of the form $[\mathbb{A}^3/G]$. We follow \cite{preygel2011thomsebastiani} which shows that the matrix factorization category has descent property to reduce the global computation to affine open covers. On the mirror curve side, we show that the mirror curve can be decomposed into local pieces, each corresponding to a piece of $[\mathbb{A}^3/G]$ on the mirror side, such that the `exceptional gluing' gives the same gluing data. We conclude the proof by induction on the number of the triangles and eventually exhibiting both sides as  homotopy limits of isomorphic diagrams.

\subsection{Overview}
The structure of this paper is as follows. In Section 2, we introduce the relevant notions and explain the definition of the mirror pair. In Section 3, we give the proof of the homological mirror equivalence for affine toric Calabi-Yau $3$-orbifolds. In Section 4, we recall the result of \cite{Pascaleff_2019} and then prove mirror equivalence for the general case.

\subsection{Acknowledgements}

The work of QB and BF is partially supported by NSFC 11831017, NSFC 11890661 and NSFC 12125101. QB would like to thank Emanuel Scheidegger for his interest and clarification on matrix factorization categories. QB and BF would like to thank Mingyuan Hu and Yuxuan Hu for many useful discussions over the years. BF would like to thank Chiu-Chu Melissa Liu, Song Yu and Zhengyu Zong for valuable discussions. 

\section{Preliminaries}

We always write $k$ for the field of complex number $\mathbb C$, especially when we want to emphasize it plays the role of the ground field. We use cohomological grading convention.

\subsection{Categorical language} 
We fix some notations and conventions on $\infty$-categories in this section. We mostly follow the setup of \cite{Pascaleff_2019}. The reader is advised to go there for motivation and detailed discussion. We work with $\mathbb{Z}/2$-graded version of dg categories throughout. For derived Morita theory of  $\mathbb{Z}/2$-dg categories, see \cite{dyckerhoff2013triangulated} and \cite{Dyckerhoff_2017}. For the connection between $2$-periodicity and $S^1$-action on categories see \cite{preygel2011thomsebastiani}.

\begin{defi} Let $\text{Mod}_k^{(2)}$ be the cocomplete stable $\infty$-category of $2$-periodic complexes of $k$-vector spaces and $\text{Mod}_k^{(2),\omega}$ its full subcategory of compact objects.
\begin{itemize}
    \item{} A $\mathbb{Z}/2$-dg category for us is a stable $\infty$-category tensored over $\text{Mod}_k^{(2),\omega}$. A cocomplete $\mathbb{Z}/2$-dg category for us is a cocomplete stable $\infty$-category tensored over $\text{Mod}_k^{(2)}$.
	\item{}  We denote $\text{DGCat}^{(2)}_{\text{small}}$ the $\infty$-category of small $\mathbb{Z}/2$-dg categories and exact functors between them. We denote $\text{DGCat}^{(2)}_{\text{cont}}$ the $\infty$-category of cocomplete $\mathbb{Z}/2$-dg categories and continuous functors between them. Note that these functors are required to be linear over $\text{Mod}_k^{(2),\omega}$ (respectively $\text{Mod}_k^{(2)}$).
	\item{} Given an object $C\in\text{DGCat}^{(2)}_{\text{cont}}$, we denote $C^\omega\in\text{DGCat}^{(2)}_{\text{small}}$ its full subcategory of compact objects. Beware that a functor in $\text{DGCat}^{(2)}_{\text{cont}}$ doesn't always send compact objects to compact objects.
	\item{} We have the Ind-completion functor:
	$$\text{Ind}:\text{DGCat}_{\text{small}}^{(2)}\rightarrow\text{DGCat}^{(2)}_{\text{cont}}$$
	 See \cite{gaitsgory2017study} for concrete definition. 
\end{itemize}
\end{defi}
\begin{rem} The above homotopy theoretic definition seems very abstract, but are good for various categorical constructions. On the other hand, we also talk about categories strictly enriched over $\mathbb{Z}/2$-graded chain complexes (or $2$-periodic chain complexes) of $k$-vector spaces and enriched functors. One can easily translate between $\mathbb{Z}/2$-graded chain complexes and $2$-periodic chain complexes so we might identify them. There is also a folding construction taking any chain complex to a $2$-periodic chain complex. We can apply this folding construction to any category enriched over chain complexes $\mathcal{C}$ and get its $\mathbb{Z}/2$-folding $\mathcal{C}_{\mathbb{Z}/2}$. For details see Section 5 of \cite{dyckerhoff2013triangulated} or section 2.1 of \cite{nadler2016wrapped}. We care about the homotopy theory of these categories, and one possible approach is the derived Morita theory of \cite{toen2007homotopy}. In fact the $\infty$-category of dg categories obtained from inverting Morita equivalences can be identified with the full subcategory of idempotent complete categories inside $\text{DGCat}^{(2)}_{\text{small}}$. See section 2.1.1. of \cite{Pascaleff_2019} for details on relating rigid model to homotopical setting above. Although we use the formalism of $\infty$-categories, most of our diagrams of $\mathbb{Z}/2$-dg categories will commute on the nose, and we can use the derived Morita model structure for computation.
\end{rem}

We will talk about (co)sheaf of dg categories on a site $\mathcal{C}$. A sheaf of dg categories is just a presheaf on $\mathcal{C}$ valued in suitable category of dg categories (which will be evident in the context, for example, $\text{DGCat}^{(2)}_{\text{cont}}$) satisfying further descent properties. We only require a sheaf to satisfy \v Cech descent (but enough for our purposes, since we only care about sheaves on graphs or nice algebraic stacks equipped with Zariski topology).
\begin{defi} Let $\mathcal{F}$ be a presheaf of dg categories on a site $\mathcal{C}$. We say that $\mathcal{F}$ satisfies \v Cech descent if for any cover $\mathcal{U}\rightarrow V$ (here $\mathcal{U}\rightarrow V$ is actually a family of arrows)
$$\mathcal{F}(V)\rightarrow[\mathcal{F}(\mathcal{U})\rightrightarrows \mathcal{F}(\mathcal{U}\underset{V}\times\mathcal{U})\triplerightarrows\mathcal{F}(\mathcal{U}\underset{V}\times\mathcal{U}\underset{V}\times\mathcal{U})\cdots] $$
realizes $\mathcal{F}(V)$ as a homotopy limit of the right hand side diagram. If this is the case, we say $\mathcal{F}$ is a sheaf of dg categories. The dual notion of cosheaf is defined by reversing all the arrows above and requiring $\mathcal{F}(V)$ to be a homotopy colimit.
\end{defi}
\begin{rem}\label{cech}
These diagrams degenerate to (co)equalizers in the following situations:
\\(1) when triple and higher intersections of the cover $\mathcal{U}$ are empty,\\
(2) when $\mathcal{F}$ evaluated on the triple and higher intersections is equivalent to zero category.\\
In either of the situation, the fact that $\mathcal{F}$ is a sheaf just says 
$$\mathcal{F}(V)\rightarrow\mathcal{F}(\mathcal{U})\rightrightarrows \mathcal{F}(\mathcal{U}\underset{V}\times\mathcal{U}) $$
is a (homotopy) equalizer. For the situation that we are interested in, either (1) or (2) will hold.
\end{rem}

Given a scheme or Deligne-Mumford stack $X$, we denote $\text{Perf}_{\mathbb{Z}/2}(X)\in\text{DGCat}^{(2)}_{\text{small}}$ the $\mathbb{Z}/2$-folding of the category of perfect complexes on $X$, we denote $\text{Coh}_{\mathbb{Z}/2}(X)\in\text{DGCat}^{(2)}_{\text{small}}$ the $\mathbb{Z}/2$-folding of the category of bounded coherent complexes on $X$. Likewise we denote $\text{QCoh}_{\mathbb{Z}/2}(X)\in\text{DGCat}^{(2)}_{\text{cont}}$ the $\mathbb{Z}/2$-folding of the category of unbounded quasicoherent complexes on $X$. All schemes or stacks in this paper are perfect stacks in the sense of \cite{benzvi2010integral}. Hence we have $$\text{QCoh}_{\mathbb{Z}/2}(X)\cong\text{Ind}(\text{Perf}_{\mathbb{Z}/2}(X))$$

\begin{rem} When $X$ is a smooth quasiprojective variety, perfect complexes and bounded coherent complexes coincides $\text{Perf}_{\mathbb{Z}/2}(X)=\text{Coh}_{\mathbb{Z}/2}(X)$. In general they are different.
\end{rem}

\subsection{Toric Calabi-Yau $3$-orbifolds}
We refer to \cite{Borisov_2004} for the construction of toric Deligne-Mumford stacks. We carry out the relevant construction in dimension $3$, first in the affine case, then for the general case.  Our notion follows \cite{fang2019genus,fang2019remodeling}.

\subsubsection*{Affine toric Calabi-Yau $3$-orbifolds} 

Let $$\mathbb{T}=(\mathbb{C}^*)^3,\ N=\Hom(\mathbb{C}^*,\mathbb{T}),\ M=\Hom(\mathbb{T},\mathbb{C}^*)=\Hom(N,\mathbb{Z}).$$ We denote $N_R=N\otimes_{\mathbb Z} R$ and $M_R=M\otimes_{\mathbb Z} R$ for $R=\mathbb C$ or $\mathbb R$. Consider a simplicial cone $\sigma\subset N_\mathbb{R}\cong \mathbb{R}^3$ spanned by $b_1,b_2,b_3\in N$, such that the simplicial affine toric variety $X_\sigma:=\spec \mathbb C[\sigma^\vee\cap M]$ has trivial canonical divisor. Thus there is a $u\in M$ with $(u,b_i)=1$. We can pick $\mathbb{Z}$-basis $\{u_i\}$ of $M$ such that $u_3=u$ and let $\{e_i\}\subset N$ be the dual $\mathbb{Z}$-basis. Under these coordinates and up to an $SL_2(\mathbb Z)$-rearrangement, $b_i$ can be written as
\begin{equation}
b_1=re_1-se_2+e_3,\,b_2=me_2+e_3,\,b_3=e_3
\label{eqn:b-rays}
\end{equation}
with positive integers $r$ and $m$, while $s\in\{0,1,\dots,r-1\}$. Such fan data induce a map on the torus
\begin{equation}
1\longrightarrow G\longrightarrow (\mathbb{C}^*)^3\stackrel{\phi}{\longrightarrow} \mathbb T \cong (\mathbb{C}^*)^3\rightarrow 1
\label{eqn:G-group}
\end{equation}
where
$$\phi(t_1,t_2,t_3)=(t_1^r,t_1^{-s}t_2^m,t_1t_2t_3).$$ 
We will write $\rho_i:G\to\mathbb{C}^*$ for the projection to the $i$-th coordinate. Then $\rho_i\in G^\vee=\Hom(G,\mathbb C^*)$ and satisfies $\rho_1\rho_2\rho_3=1$. Moreover $\rho_i$ induces an exact sequence
\[
1\longrightarrow \mu_{m_i}\longrightarrow   G \stackrel{\rho_i}{\longrightarrow} \mu_{r_i}\longrightarrow 1,
\]
where $\mu_k\cong \mathbb Z/(k \mathbb Z)$ denotes the subgroup of $k$-th roots of unity in $\mathbb{C}^*$. One can compute 
\begin{align*}
	& (m_1,r_1)=(m,r), \\
	& (m_2,r_2)=(\gcd(r,s),m\cdot r/\gcd(r,s)), \\
	& (m_3,r_3)=(\gcd(m+s,r),m\cdot r/\gcd(m+s,r)).
\end{align*}
The associated affine toric orbifold is $\mathcal{X}(\sigma):=[\mathbb{A}^3/G]$, with coarse moduli $X_\sigma$.

\subsubsection*{General case}

With the notions introduced in the affine case, we consider a convex simplicial fan $\Sigma$ in $N_\mathbb{R}$. Let $\Sigma(d)$ be the set of $d$-dimensional cones in $\Sigma$, we assume $\Sigma(3)\neq\emptyset$. Let $\Sigma(1)=\{\rho_1,\cdots,\rho_{3+p'}\}$ be the set of $1$-cones in $\Sigma$ and $b_i$ be the generator of $\rho_i\cap N\cong\mathbb{Z}_{\geq 0}$. 

Similar to the affine case, by the Calabi-Yau condition, we may choose $b_i=(m_i,n_i,1)$. We think of these coordinates $(m_i,n_i)$ as defining a triangulated convex polytope on the plane placed at height $1$ and $\Sigma$ is the cone over it.

From this combinatorial data, we can define a stacky fan $\Sigma^{can}=(N,\Sigma,\beta^{can}=(b_1,\cdots,b_{3+p'}))$. By \cite{Borisov_2004} there is a toric CY three orbifold associated to this stacky fan $\mathcal{X}(\Sigma)=[U_\Sigma/G_\Sigma]$. For each simplicial cone $\sigma\in\Sigma(3)$, there is an open substack $\mathcal{X}(\sigma)\cong[\mathbb{A}^3/G]$, which is described by the affine case above after putting it in favorable coordinates. Hence we might think of $\mathcal{X}(\Sigma)$ as glued together from $\mathcal X(\sigma)$ for all $\sigma\in\Sigma(3)$. We use $\mathcal X=\mathcal X(\Sigma)$ to denote the toric Calabi-Yau $3$-orbifold we want to study under mirror symmetry throughout the paper.

\subsection{Category of Matrix Factorizations}
Consider the category of bounded  complexes of coherent sheaves $\coh(\mathcal{X})$. Note that there is canonical equivalence 
$$\coh^G(X)=\coh([X/G])$$
and the left hand side is the bounded derived category of $G$-equivariant coherent sheaves on $X$. Recall that a $G$-equivariant coherent sheaf is the data of a coherent sheaf on $X$ equipped with a $G$-linearization. In this note we often use these two categories interchangeably for convenience of phrasing. To compute mapping complex in this category, note that for a finite group $G$ acting on a variety $X$ and equivariant coherent sheaves $E,F$ on $X$, $\ext^n_X(E,F)$ naturally carries a $G$-action, and we can compute $\ext^n_{[X/G]}(E,F)$ as
$$\ext^n_{[X/G]}(E,F)=\ext^n_X(E,F)^G$$
We now define two variants of $\coh(\mathcal{X})$. Consider a pair $(\mathcal{X},f)$ with $\mathcal{X}$ a smooth DM stack and $f:\mathcal{X}\rightarrow \mathbb{A}^1$ flat morphism, smooth away from origin. In this note, $\mathcal{X}$ will be our toric orbifold, and $f$ is the function given by $u\in M$ above trivializing the canonical bundle. Many categories were constructed to capture the geometry of the singular fiber $f^{-1}(0)=\mathcal{X}_0$. The easier one is the category of singularity.
\begin{defi} Let $\mathcal{X}_0:=f^{-1}(0)$. $D_{Sing}(\mathcal{X}_0)\in \text{DGCat}^{(2)}_{\text{small}}$ is defined to be the dg quotient $\text{Coh}(\mathcal{X}_0)/\text{Perf}(\mathcal{X}_0)$ of bounded coherent complexes by perfect complexes. This $2$-periodic dg category is called the category of singularity.
\end{defi}

For affine schemes $\spec A$ and a function $f\in A$, Orlov \cite{Or04} defined a triangulated category of matrix factorizations. It has as objects pairs of finitely generated projective $A$-modules $(E_0,E_1)$ equipped with maps $e_0:E_0\rightarrow E_1$ and $e_1:E_1\rightarrow E_0$ such that compositions $e_0e_1$ and $e_1e_0$ are equal to multiplication by $f$. Morphisms are maps of pairs commuting with structure maps modulo homotopy. It was later generalized to the nonaffine and equivariant setting \cite{Orlov_2011}. We will use a dg enhenced version of it as in \cite{preygel2011thomsebastiani}.
\begin{defi} We define MF$(\mathcal{X},f)\in \text{DGCat}^{(2)}_{\text{small}}$ to be the $\mathbb{Z}/2$-dg category of matrix factorizations. For further use, we denote $\text{MF}^\infty(\mathcal{X},f)\in \text{DGCat}^{(2)}_{\text{cont}}$ its Ind-completion.
\end{defi}
This category of matrix factorization is much more computable. These two categories are related by the following, hence we may translate between them.
\begin{prop}[\cite{Or04}] There is a functor $\text{MF}(\mathcal{X},f)\rightarrow D_{Sing}(\mathcal{X}_0)$ taking a pair $(E_0,E_1)\in$ MF$(\mathcal{X},f)$ to $\text{coker}(e_0)\in D_{Sing}$ viewed as a coherent sheaf on $\mathcal{X}_0$. This is an equivalence of $2$-periodic dg categories.
\end{prop}
The theorem was written in the language of triangulated categories, and extends naturally to dg enhencements. See also \cite{abouzaid2014homological} Section 7 for a similar detailed discussion in the equivariant setting.
\begin{prop}[\cite{preygel2011thomsebastiani}] Proposition\, A.3.1
The assignment $U\mapsto\text{MF}^\infty(U,f|_U)$ is a sheaf of $\mathbb{Z}/2$-dg categories under \'etale topology.

\end{prop}

\subsection{Mirror curve construction}
\label{sec:mirror-curve}

The mirror geometry is described by a non-compact Riemann surface. We first describe the affine situation, then more generally. We follow the notion in \cite{fang2019genus,fang2019remodeling}. In particular, we will describe the action of the relevant stacky Picard groups on the mirror curve.

\subsubsection{Affine case} Recall that the combinatorial data for affine toric Calabiy-Yau $3$-orbifold $\mathcal X=[\mathbb C^3/G]$ are just three nonnegative integers $(m,r,s)$ in Equation \eqref{eqn:b-rays}. We define an affine curve 
$$\mathcal{C}_\sigma:=\{H(X,Y)=X^rY^{-s}+Y^m+1=0\}\subset(\mathbb{C}^*)^2.$$
This affine curve has a natural compactification $\overline{\mathcal{C}}_\sigma$ in the projective toric orbifold surface $\mathbb P_{\Delta_\sigma}$ defined by the moment polytope $\Delta_\sigma$ with vertices $0,0),(0,m),(r,-s)$.

The curve $\mathcal C_\sigma$ is an exact symplectic manifold carrying a $G^\vee$-action. We set 
$$w_1=\frac{1}{r}, w_2=\frac{s}{rm}, w_3=-w_1-w_2.$$
There are two distinguished elements in $G\subset (\mathbb C^*)^3$ (c.f. Equation \ref{eqn:G-group})
$$\eta_1=(e^{2\pi\sqrt{-1}w_1},e^{2\pi\sqrt{-1}w_2},e^{2\pi\sqrt{-1}w_3}), \eta_2=(1,e^{2\pi\sqrt{-1}/m},e^{-2\pi\sqrt{-1}/m}).$$
such that they define a bijection of sets (but not necessarily group homomorphism) $\mathbb{Z}/r \times \mathbb{Z}/m\cong G$. For each $\chi\in G^\vee$, the action on the ambient space
$$\chi\cdot(X,Y)=(\chi(\eta_1)X,\chi(\eta_2)Y)$$
gives rise to an action on $\mathcal{C}_\sigma$. Another way to interpret this group action on $\mathcal C_\sigma$ is by the notion of stacky Picard group \cite{ccit2020}. As discussed in \cite[Section 4.7]{fang2019remodeling}, one may identify $G^\vee$ with $\mathrm{Pic}^{\mathrm{st}}(\mathcal X)=\Pic(\mathcal X)/\Pic(X)$, and then the action of a line bundle class $[\mathcal L]\in \Picst(\mathcal X)$ on $\mathcal C_\sigma$ is given in coordinates
\[
[\mathcal L]\cdot X=\exp(-2\pi\sqrt{-1}\mathrm{age}_{(1,0,1)}(\mathcal L)) X,\ [\mathcal L]\cdot Y=\exp(-2\pi\sqrt{-1}\mathrm{age}_{(0,1,1)}(\mathcal L)) Y.
\]

\begin{prop}\label{quot}
The action of $G^\vee$ on $\mathcal{C}_\sigma$ is free, and it extends to an action of $G^\vee$ on $\overline{\mathcal C}_\sigma$. The quotient of $\mathcal{C}_\sigma$ under this action can be identified with three-punctured sphere $\mathbb{P}^1-\{0,1,\infty\}$, while the quotient of $\overline{\mathcal C}_\sigma$ is $\mathbb P^1$. The quotient map $p:\overline{\mathcal C}_\sigma\to \mathbb P^1$ is a ramified covering. The preimages of $\{0,1,\infty\}$ are repsectively the intersection of $\overline{\mathcal{C}}_\sigma$ with three toric divisors of $P_{\Delta_\sigma}$. The group $G^\vee$ permutes transitively inside each $p^{-1}(0)$, $p^{-1}(1)$, $p^{-1}(\infty)$.
\end{prop}
\begin{proof} See \cite[Section 6.4]{fang2019genus} for details. The three edges of the triangle $\Delta_{\sigma}$ define three toric divisor $E_i$ and $\mathcal{C}_\sigma$ is obtained from $\overline{\mathcal{C}}_\sigma$ by removing the intersection $\overline{\mathcal{C}}_\sigma\cap E_i$. The number $\#(\overline{\mathcal{C}_\sigma}\cap E_i)+1$ is the number of the lattice point on corresponding edge, which is exactly $m_i+1$ as we defined earlier. Direct computation shows that this action is free, and permutes inside each $\overline{\mathcal{C}_\sigma}\cap E_i$ transitively. That the quotient is a three-punctured sphere is an easy application of Hurwitz formula.
\end{proof}

\begin{rem}
	\label{rmk:quotient} 
    The quotient map $p:\mathcal C_\sigma \to \mathbb P^1-\{0,1,\infty\}$ is given by 	
	\[
	(X,Y)\mapsto (X^rY^{-s},Y^m),
  	\]
	where the image is identified with the curve $1+X+Y=0$ for $X,Y\in \mathbb C^*$. 
\end{rem}

The curve $\mathcal C_\sigma$ has another compactification $\widetilde{\mathcal C_\sigma}$ by attaching an $S^1$ to each of its puncture. We only consider this compactification topologically. We denote $\widetilde{\mathcal C_\tau}$ to be the collection of $S^1$ in $\partial \widetilde{\mathcal C_\sigma}$ for each $2$-cone $\tau$ as the face of $\mathcal \sigma$, so that $\partial \widetilde{\mathcal C_\sigma}=\cup_{\tau\subset \sigma} \widetilde{\mathcal C_\tau}$. The action of $\Picst(\mathcal X_\sigma)=G^\vee$ on $\mathcal C_\sigma$ extends to an action on $\widetilde{\mathcal C_\sigma}$. 

\subsubsection{General case}
\label{sec:general-mirror-curve} Recall that the combinatorial data needed to produce a general Calabi-Yau 3-orbifold is a fan $\Sigma$ constructed from taking cone over a triangulated planar polygon $\Delta$. From these data we can cook up a polynomial $H(X,Y,q)$ and define the mirror curve $\mathcal{C}=\{H(X,Y,q)=0\}\subset(\mathbb{C}^*)^2$.

We follow the notion of \cite{fang2019remodeling}. The complex parameter $q\in \mathbb C^{\mathfrak p}$, where the dimension of the extended K\"ahler moduli $\mathfrak p=\mathrm{dim} H^2_{\mathrm{CR}}(\mathcal X)$ is the number of the integer points in $\Delta$ minus $3$. We write $q=(q_K,q_{\mathrm{orb}})$, where $q_K$ is the first $\mathfrak p'=\mathrm{dim} H^2(\mathcal X)$ coordinates and $q_{\mathrm{orb}}$ is the remaining $\mathfrak p-\mathfrak p'$ coordinates. The mirror curve has a compactification $\overline{\mathcal C}$ in the toric orbifold surface $\mathbb P_\Delta$ prescribed by its moment polytope $\Delta$.

As $q\to 0$, the compactified mirror curve $\overline{\mathcal C}$ degenerates into
\[
\overline{\mathcal C_0}=\bigcup_{\sigma \in \Sigma(3)} \overline{\mathcal C_{0,\sigma}},	
\]
and each $\mathcal C_{0,\sigma}\cong \mathcal C_\sigma$ is just the mirror curve of the corresponding affine toric Calabi-Yau $3$-orbifold $\mathcal X_\sigma$. The action of the stacky Picard group $\Picst(\mathcal X)$ on $\overline{\mathcal C}_0$ preserves each component. Let $\sigma$ and $\sigma'$ be two adjacent $3$-cones and $\tau = \sigma \cap \sigma'$ be their non-empty intersection $2$-cone. The action of $\Picst(\mathcal X)$ on $\mathcal C_{0,\sigma}$ (resp. $\mathcal C_{0,\sigma'}$, $\mathcal C_{0,\tau}$) factors through $\Picst(\mathcal X_\sigma)$ (resp. $\Picst(\mathcal X_{\sigma'})$, $\Picst(\mathcal X_\tau)$) under the following diagram
\[
   \begin{CD}
		\Picst(\mathcal X) @>>> \Picst(\mathcal X_\sigma)=G_\sigma^\vee\\
		@VVV @VVV\\
		\Picst(\mathcal X_{\sigma'})=G_{\sigma'}^\vee @>>> \Picst(\mathcal X_\tau)=G_{\tau}^\vee.
   \end{CD}
\]
Here $\mathcal C_{0,\tau}=\overline{\mathcal C_{0,\sigma}}\cap \overline{\mathcal C_{0,\sigma'}}$ is a set of $|G_\tau|$ points.

We set $q_{\mathrm{orb}}=0$ and $q_K$ to be positive and very small. The resulting mirror curve $\mathcal C$ is smooth, and allows the following decomposition
\[
\widetilde{\mathcal C} =\bigcup_{\sigma\in \Sigma(3)} \widetilde{\mathcal C_\sigma^\circ},	
\]
where each component is identified with the mirror curve of the affine toric Calabi-Yau $3$-orbifold $\mathcal X_\sigma$ via a diffeomorphism $u_\sigma:\mathcal C^\circ_\sigma \cong \mathcal C_\sigma$. The action of $\Picst(\mathcal X)$ preserves each $\mathcal C^\circ_\sigma$ (and $\widetilde{\mathcal C^\circ_\sigma}$) but does not necessarily factors through $\Picst(\mathcal X_\sigma)$. We use the identification $u_\sigma$ to carry the action of $\mathcal C_\sigma$ on $\mathcal C_\sigma^\circ$ (and to the action on its extenstion to $\widetilde{\mathcal C_\sigma^\circ}$). From here we always talk about the action on $\mathcal C_\sigma^\circ$ and $\widetilde{\mathcal C_\sigma^\circ}$ via the identification $u_\sigma$.

For two adjacent $3$-cones $\sigma,\sigma'$ with $\tau=\sigma \cap \sigma'$, $\widetilde{\mathcal C_\tau}=\widetilde{\mathcal C^\circ_{\sigma}}\cap \widetilde{\mathcal C^\circ_{\sigma'}}$ is a disjoint union of $|G^\vee_\tau|$ circles. By taking limit, the action of $\Picst(\mathcal X_\sigma)$ (resp. $\Picst(\mathcal X_{\sigma'})$) extends to $\widetilde{\mathcal C_\tau}$. We label these circles (connected components of $\widetilde{\mathcal C_\tau}$) by $G^\vee_\tau$ such that the action of $\Picst(\mathcal X_\tau)$ (via its lift to $\Picst(\mathcal X_\sigma)$ or $\Picst(\mathcal X_{\sigma'})$) permutes these labelings by group multiplication. Notice that the group $\Picst(\mathcal X_\tau)$ does not act on $\widetilde {\mathcal C_\tau}$, only on the set of its connected components. Similary for a $2$-cone $\tau$ in a single $3$-cone $\sigma$, we also choose a labeling of the connected components of $\widetilde{\mathcal C_\tau}$ by $G^\vee_\tau$ and then $G^\vee_\tau$ acts on it by multiplication.


\begin{rem}
We also record here a combinatorial description of the mirror curve taken from Section 5.4 of \cite{fang2019remodeling} and refer to Figure 4 there for illustration and example. We mark each triangle by the cone $\sigma$ over it and think of the mirror curve $\mathcal{C}_\sigma$ sits inside of this triangle, with corresponding punctures living on the corresponding edges. Then the global curve $\mathcal{C}_\Sigma$ can be described as gluing these $\mathcal{C}_\sigma$ together along the punctures on the same edge.
\end{rem}

\subsection{Topological Fukaya category}
The topological Fukaya category is a cosheaf of dg categories designed to model the wrapped Fukaya category of a punctured Riemann surface $\mathcal{C}$. The definition is due to \cite{dyckerhoff2013triangulated}. Here we follow the recent treatment of \cite{Pascaleff_2019}. We collect the definition, some useful properties and refer to the original work for detailed construction. First we recall the generalities about ribbon graphs. 
\begin{defi}A graph $\Gamma$ is the datum of a pair of finite sets $(H,V)$ equipped with an involution $\tau:H\rightarrow H$ and a map $s:H\rightarrow V$. The elements in $H$ are called half-edges and elements in $V$ are vertices. The set of edges $E$ is the  conjugacy classes of half-edges under $\tau$. The preimage $s^{-1}(v)$ is the set of half-edges adjacent to $v$. 
\end{defi}
We allow graphs to have edges with only one end, i.e. they are invariant under $\tau$. These edges are called external edges while others are internal edges. We can subdivide an edge by adding a vertex at the middle, and we often assume that a graph is sufficiently subdivided. In fact, subdividing won't change the construction of topological Fukaya category. By an open subgraph $\Gamma'\subset \Gamma$, we mean a subgraph that if a vertex $v$ lies in $\Gamma'$, all the half-edges adjacent to it will also lie in $\Gamma'$. Note that open subgraphs are closed under finite intersections and unions.
\begin{defi} A ribbon graph $\Gamma$ is the datum of a graph with a cyclic order at each vertex for the set of the edges incident to it.
\end{defi}
One can interpret the ribbon graph structure as follows. For a graph $\Gamma$, we define the incidence category $I(\Gamma)$. It has objects $V \cup E$ and for each internal halfedge $h$ a morphism from $s(h)$ to $[h]$. We have a presheaf $\gamma:I(\Gamma)\rightarrow Set$ which, on objects, takes $v\in V$ to $s^{-1}(v)$ and $[h]\in E$ to $\{h,\tau(h)\}$. For a morphism $v\rightarrow [h]$ given by $h\in H$, we define map $H(v)\rightarrow\{h,\tau(h)\}$ taking $h$ to $h$ and the rest to $\tau(h)$. Here $\gamma$ is called the incidence diagram of the graph $\Gamma$.
\begin{prop} A ribbon graph structure on $\Gamma$ is equivalent to a lift of the functor $\gamma$ to $I(\Gamma)\rightarrow\Lambda$, where $\Lambda$ is the category of finite cyclically ordered sets.
\end{prop}
Making use of the interstice duality (\cite{Dyckerhoff_2017} Proposition 1.8) $\Lambda\cong\Lambda^{op}$, we obtain two functors $\gamma:I(\Gamma) \rightarrow \Lambda$ and $\delta:I(\Gamma) ^{op} \rightarrow\Lambda$.
\begin{defi} We introduce the category $Rib$ of ribbon graphs as follows. A morphism $(f,\eta)$ of ribbon graphs $\Gamma\rightarrow\Gamma'$ consists of a functor $f:I(\Gamma)\rightarrow I(\Gamma')$ and a natural transformation $\eta:f^*\gamma'\rightarrow\gamma$. 
\end{defi}
We also consider the category $Rib^*$ of pointed ribbon graphs whose objects are $(\Gamma,x)$ where $\Gamma$ is a ribbon graph and $x$ is an object of $I(\Gamma)$. A morhpism $(\Gamma,x)\rightarrow(\Gamma',y)$ consists of a morphism $(f,\eta)$ of ribbon graphs and a morphism $y\rightarrow f(x)$ in $I(\Gamma')$. There is a forgetful functor $\pi:Rib^*\rightarrow Rib$ and an evaluation functor $ev:Rib^*\rightarrow \Lambda:ev((\Gamma,x))=\delta(x)$. From this we can cook up a state sum functor.
\begin{defi}
For any $\infty$-cateogry $C$ with all colimits and $X:N(\Lambda)\rightarrow C$ a cocyclic object in $C$, we call the following functor state sum of $X$ on $\Gamma$,
$$\rho_X=N(\pi)_!(X\circ N(ev)):N(Rib)\rightarrow C$$
where, to be precise, we take nerve of the $1$-categories to view it as an $\infty$-category, and $N(\pi)_!$ is the $\infty$-categorical left Kan extension. Note that by the formula of left Kan extension, we have
$$\rho_X(\Gamma)=X(\Gamma)=\text{colim}X\circ N(\delta)$$
where the homotopy colimit is taken over $I(\Gamma)^{op}$, the opposite of incidence category.
\end{defi}
We will plug in a specific cocyclic object $\mathcal{E}^*$ to obtain the definition of topological Fukaya category.
\begin{prop} (\cite{Dyckerhoff_2017} Theorem 3.2) There is a 2-Segal cocyclic object $\mathcal{E}^*:N(\Lambda)\rightarrow \text{DGCat}^{(2)}_{\text{small}}$ whose value $\mathcal{E}^n$ can be computed as the $\mathbb{Z}/2$-folding of the dg category of perfect modules over $A_{n-1}$ quiver.
\end{prop}
\begin{defi} We call the state sum of $\mathcal{E}^*$ on a ribbon graph $\Gamma$ the topological Fukaya category of the ribbon graph $\mathcal{F}^{top}(\Gamma)\in \text{DGCat}^{(2)}_{\text{small}}$. We write $\mathcal{F}^{top}_\infty(\Gamma)\in \text{DGCat}^{(2)}_{\text{cont}}$ for its ind completion.
\end{defi}
\begin{rem} (1) If one dualizes above construction, the output is called the compact Fukaya cateogry of the ribbon graph $\mathcal{F}_{top}(\Gamma)$. As we will never use it, we won't repeat the dual construction here. See Definition 4.10 of \cite{Pascaleff_2019}.\\
 (2) Here we follow the convention of \cite{Pascaleff_2019}. In the original work of \cite{dyckerhoff2013triangulated} these two categories are called topological Fukaya and coFukaya category to distinguish their functoriality. The current terminology is picked to reflect that they model different kinds of Fukaya categories from symplectic geometry.\\
(3) There are other possbile constructions of the topological Fukaya category. Nadler in \cite{nadler2016wrapped} constructed a combinatorial Fukaya category using microlocal sheaf theory. In the case of cotangent bundle, his wrapped microlocal sheaf category $\mu Sh_\Gamma^w(-)$ behaves as a cosheaf and is a candidate for the local model of the topological Fukaya category in higher dimension.
\end{rem}
The following properties follow from the construction and characterize the topological Fukaya categories:
\begin{prop} \label{glue}(\cite{Pascaleff_2019} Section 5) $\mathcal{F}^{top}$ and $\mathcal{F}_{top}$ sastisfy the following properties:\\
(1) Given an open subgraph $U\subset\Gamma$, we have a corestriction functor and a restriction functor $$C_U:\mathcal{F}^{top}(\Gamma)\leftarrow\mathcal{F}^{top}(U);\,R_U:\mathcal{F}_{top}(\Gamma)\rightarrow\mathcal{F}_{top}(U) $$
(2) Given two open subgraphs $U$ and $V$ such that $U\cup V=\Gamma$, we have a pushout diagram
and a pullback diagram in $\text{DGCat}_{\text{small}}^{(2)}$
\begin{center}
\begin{tikzcd}
\mathcal{F}^{top}(U\cap V)\arrow{r}{C_{U\cap V}}\arrow{d}{C_{U\cap V}}&\mathcal{F}^{top}(U)\arrow{d}{C_{U}}\\
\mathcal{F}^{top}(V)\arrow{r}{C_{V}}&\mathcal{F}^{top}(U \cup V)
\end{tikzcd}
\begin{tikzcd}
\mathcal{F}_{top}(U\cap V)&\mathcal{F}_{top}(U)\arrow{l}{R_{U\cap V}}\\
\mathcal{F}_{top}(V)\arrow{u}{R_{U\cap V}}&\mathcal{F}_{top}(U\cup V)\arrow{l}{R_V}\arrow{u}{R_U}
\end{tikzcd}
\end{center}
(3) For the ribbon graph $\Gamma^n$ consisting of only one vertex and $n$ external edges attached to it equipped with arbitrary cyclic order, one can compute $$\mathcal{F}^{top}(\Gamma^n)\cong\text{Perf}_{\mathbb{Z}/2}(A_{n-1})$$
(4) We have a duality between these two types of topological Fukaya categories
$$\mathcal{F}_{top}(\Gamma)\cong \text{Fun}(\mathcal{F}^{top}(\Gamma),\text{Mod}_k^{(2),\omega})$$
\end{prop}

To relate to the surface topology, we have the following definition.
\begin{defi} An embedded graph $\Gamma\subset \mathcal{C}$ in a not-necessarily-compact Riemann surface is called a spine or skeleton if $\mathcal{C}$ strongly deformation retracts to $\Gamma$.
\end{defi}
Note that an embedded grpah in a Riemann surface naturally acquires a ribbon graph structure from the orientation of the ambient space. We would want to assign to the surface the topological Fukaya category of its skeleton, provided it is well defined. This is indeed the case.

\begin{prop}\label{invariance}\cite{dyckerhoff2013triangulated} The assignment $\mathcal{F}^{top}(\Gamma)$, $\mathcal{F}_{top}(\Gamma)$ and $\mathcal{F}^{top}_\infty(\Gamma)$ are invariants of the surface $\mathcal{C}$ with specified structure near infinity. In fact, for any other skeleton $\Gamma'$ of $\mathcal{C}$ with same structure near infinity as $\Gamma$, we have a canonical equivalence of categories
$$\Psi:\mathcal{F}^{top}(\Gamma)\cong\mathcal{F}^{top}(\Gamma').$$
\end{prop}
\begin{rem} 
	This ``structure near infinity'' models the stop in a \emph{partially-wrapped} Fukaya category in the sense of \cite{Ganatra_2019}. We do not give a detailed definition here -- for our purpose, we should consider the (fully) wrapped Fukaya category on our mirror curve. Thus when we talk about the topological Fukaya category of a curve $\mathcal{C}$, we define $\mathcal{F}^{top}(\mathcal{C}):=\mathcal{F}^{top}(\Gamma)$ for some \emph{compact} skeleton $\Gamma$ of the curve, which corresponds to empty structure near infinity. By the above proposition, this is a well-defined topological invariant of the curve.
\end{rem}
\begin{example} \label{wheel}  The cotangent bundle $T^*S^1$ has skeleton of the form $\Gamma(p,q)$ called wheels. Each wheel has a central circle $S^1$ embedded as the zero section and $p+q$ fibers called spokes attached to it, with $(p,q)$ counts the number of upward pointing and downward pointing conic cotangent fibers. One can compute (\cite{Pascaleff_2019} Lemma 6.3)
\begin{itemize}
	\item{$\mathcal{F}^{top}(\Gamma(p,q))\cong\text{Perf}_{\mathbb{Z}/2}(\mathbb{P}^1(p,q))$, $p,q$ both nonzero;}
	\item{$\mathcal{F}^{top}(\Gamma(n,0))\cong\text{Perf}_{\mathbb{Z}/2}([\mathbb{A}^1/\mu_n])$, $n$ nonzero;}
	\item{$\mathcal{F}^{top}(\Gamma(0,0))\cong\text{Perf}_{\mathbb{Z}/2}(\mathbb{G}_m$).}
	\end{itemize}
For further use, we record here a quiver representation model for the topological Fukaya category of wheels. Consider a family of quivers $Q(\Gamma(p,q))$ defined for each $\Gamma(p,q)$. It has vertices indexed by the closed intervals separated by spokes on the central circle of $\Gamma(p,q)$, while the arrows are indexed by the spokes. If a spoke is pointing upwards, it gives an arrow from the vertex right before it to the vertex right after it, and vice versa. There is an equivalence $$\mathcal{F}^{top}(\Gamma(p,q))=\text{Perf}_{\mathbb{Z}/2}(Q(\Gamma(p,q)))$$ 
(cf. \cite{sibilla2011ribbon} Theorem 3.18).
\end{example}
\begin{rem}
Strictly speaking, the spokes in $\Gamma(p,q)$ can be arranged in different positions, and they produce different $Q(\Gamma(p,q))$. But according to Proposition \ref{invariance}, for skeletons with fixed structure near infinity, their topological Fukaya categories can be identified. Hence we will still call them $\Gamma(p,q)$ when it doesn't cause confusion.
\end{rem}
\begin{figure}
 \centering  \includegraphics[width=1\textwidth]{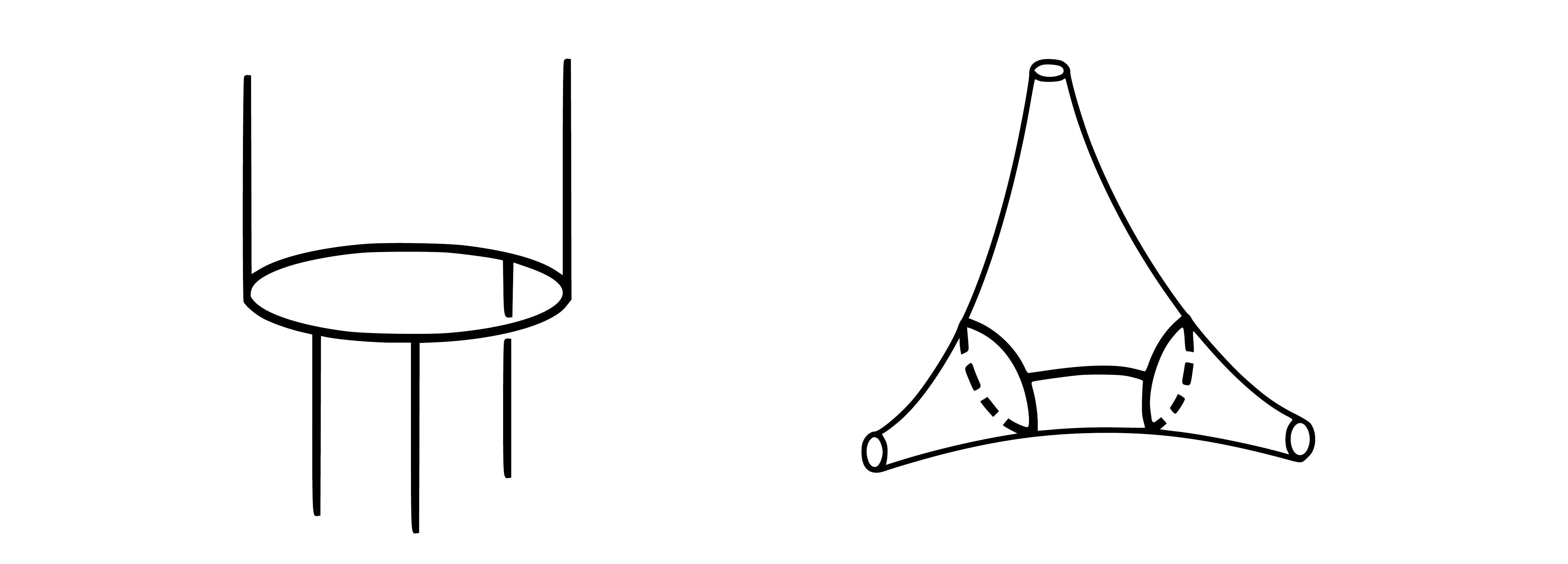} \caption{$\Gamma(2,3)$ and dumbbell} 
\end{figure}
\begin{example}\label{dumbbell} The three-punctured sphere $\mathbb{P}^1-\{0,1,\infty\}$ has a skeleton $\Gamma_D$ called dumbbell. We can decompose the dumbbell as two wheels $\Gamma(1,0)$ glued together along a common spoke. By above example we know $\mathcal{F}^{top}(\Gamma(1,0))\cong \coh_{\mathbb{Z}/2}(\mathbb{A}^1)$, and $\mathcal{F}^{top}(-)$ evaluated on a segment is by definition $\coh_{\mathbb{Z}/2}(*)$.
By cosheaf property of $\mathcal{F}^{top}(-)$ we know that $\mathcal{F}^{top}(\Gamma_D)$ is the homotopy pushout of the following diagram, with the maps being pushforward along the inclusion of origin.
$$\coh_{\mathbb{Z}/2}(\mathbb{A}^1)\leftarrow\coh_{\mathbb{Z}/2}(*)\rightarrow\coh_{\mathbb{Z}/2}(\mathbb{A}^1)$$
It can be computed to be $\coh_{\mathbb{Z}/2}(\spec k[x,y]/xy=0)$, using descent property of $\coh_{\mathbb{Z}/2}(-)$.
\end{example}

\section{Affine Case}
In this section, we prove the mirror equivalence for $(\mathcal{X}_\sigma=[\mathbb{A}^3/G],W=z_1z_2z_3)$. We work with a fixed $\sigma$ in this section. The first step is to establish an equivariant version of the dimensional reduction equivalence.

For $\mathcal{X}=[\mathbb{A}^3/G]$, we define $\mathcal{Y}=[Y/G]$ 
where
\[
	Y=\AandA=\spec(k[z_1,z_2]/(z_1z_2)),
\]
and the action of $G$ on $Y\subset \mathbb A^3$ is induced from the action on $\mathbb A^3$. Let $u,v$ be morphisms
\begin{align*}
	Y=\AandA \stackrel{u}{\leftarrow} & (z_1z_2)^{-1}(0)  \stackrel{v}{\hookrightarrow} W^{-1}(0),\\
	(z_1,z_2)  \mapsfrom &(z_1,z_2,z_3)   \mapsto (z_1,z_2,z_3).
\end{align*}
The following proposition is the equivariant version of \cite[Proposition 2.3]{nadler2016wrapped}.
\begin{prop}\label{affinecomputation}
	The functor $v_*\circ u^*$ for $G$-equivariant sheaves defines an equivalence of $\mathbb Z/2$-dg categories 
    $${\coh}_{\mathbb{Z}/2}(\mathcal{Y})\cong D_{Sing}(\mathcal{X}_0).$$
\end{prop}

\begin{proof}
We will explicitly find a set of generators on both sides and then identify their mapping spaces. By generators of $C$ we mean a subset of objects $S$ such that the smallest stable subcategory  containing $S$ and closed under taking retraction is $C$. The category $\coh(\mathcal{Y})$ is generated by $\{\mathcal{O}_{\mathbb{A}^1_j}(\theta)\}_{j=1,2}$, where $$\mathcal{O}_{\mathbb{A}^1_j}=k[z_1,z_2]/z_j$$ is viewed as a coherent sheaf carrying the induced $G$-linearization. Under the functor $v_*\circ u^*$, they go to $\{\mathcal{O}_{\mathbb{A}^2_j}(\theta)\}_{j=1,2}$ in $D_{Sing}([W^{-1}(0)/G])$, where $$\mathcal{O}_{\mathbb{A}^2_j}:=k[z_1,z_2,z_3]/z_j.$$ In fact, the objects $\{\mathcal{O}_{\mathbb{A}^2_j}(\theta)\}_{j=1,2}$ generate the category of singularity. Note that $\{\mathcal{O}_{\mathbb{A}^2_j}(\theta)\}_{j=1,2,3}$ generate $\coh(\mathcal{X}_0)$, and we will see below that $\{\mathcal{O}_{\mathbb{A}^2_3}(\theta)\}$ are generated by $\{\mathcal{O}_{\mathbb{A}^2_j}(\theta)\}_{j=1,2}$. 

To identify the corresponding mapping complexes, we notice that in $\coh(\mathcal{Y})$, $\mathcal{O}_{\mathbb{A}^1_j}(\theta)$ is resolved by
\begin{center}
\begin{tikzcd}\cdots\arrow{r}&\mathcal{O}_Y(\rho_j^2\rho_{3-j}\theta)\arrow{r}{z_j}&\mathcal{O}_Y(\rho_j\rho_{3-j}\theta)\arrow{r}{z_{3-j}}&\mathcal{O}_Y(\rho_j\theta)\arrow{r}{z_j}&\mathcal{O}_Y(\theta).
\end{tikzcd}
\end{center}
This resolution extends infinitely to the left, with $(-2n)$-th term given by $\mathcal{O}_Y(\rho_j^n\rho_{3-j}^n\theta)$ and $(-2n-1)$-th term given by $\mathcal{O}_Y(\rho_j^{n+1}\rho_{3-j}^n\theta)$. We carry along the $G$-linearization all the time to compute the $G$ action on $\ext_Y^n$
\begin{align*}
\ext^*_\mathcal{Y}(\mathcal{O}_{\mathbb{A}^1_j}(\theta),\mathcal{O}_{\mathbb{A}^1_j}(\theta'))=\ext^*_Y(\mathcal{O}_{\mathbb{A}^1_j}(\theta),\mathcal{O}_{\mathbb{A}^1_j}(\theta'))^G=k[z_{3-j},v]/(z_{3-j}v=0)(\theta^{-1}\theta')^G,\\
\ext^*_\mathcal{Y}(\mathcal{O}_{\mathbb{A}^1_j}(\theta),\mathcal{O}_{\mathbb{A}^1_{3-j}}(\theta'))=\ext^*_Y(\mathcal{O}_{\mathbb{A}^1_j}(\theta),\mathcal{O}_{\mathbb{A}^1_{3-j}}(\theta'))^G=u_j\cdot k[v](\theta^{-1}\theta')^G,
\end{align*}
in which the degree $|v|=2$, and $|u_j|=1$. Using the identity $\rho_1\rho_2\rho_3=1$, the $G$-actions on $u_j,v$ are given by $gv=\rho_3(g)v$, while $gu_j=\rho_j^{-1}(g)u_j$. The compositions of these Ext-groups are self-evident, noting that $u_1u_2=u_2u_1=v$. 

On the other hand, we want to compute corresponding mapping complexes in $D_{Sing}(\mathcal X_0)=\MF(\mathcal{X},W)$. Proposition 1.21 of \cite{Or04} allows us to  to compute the underlying $\ext$-groups: for $N>2$, we have
$$H^0\Hom_{D_{Sing}}(\mathcal{O}_{\mathbb{A}^2_j}(\theta),\mathcal{O}_{\mathbb{A}^2_{j'}}(\theta')[N])=\ext^N_{X_0}(\mathcal{O}_{\mathbb{A}^2_j}(\theta),\mathcal{O}_{\mathbb{A}^2_{j'}}(\theta'))^G.$$
The sheaf $\mathcal{O}_{\mathbb{A}^2_j}(\theta)$ on $\mathcal{X}_0$ has a stable $2$-periodic resolution
\begin{center}
\begin{tikzcd}\cdots\arrow{r}&\mathcal{O}_{X_0}(\rho_j\theta)\arrow{r}{z_j}&\mathcal{O}_{X_0}(\theta)\arrow{r}{z_3z_{3-j}}&\mathcal{O}_{X_0}(\rho_j\theta)\arrow{r}{z_j}&\mathcal{O}_{X_0}(\theta).
\end{tikzcd}
\end{center}
We can compute its cohomology as graded modules:
\begin{align*} 
H^{\text{even}}\Hom(\mathcal{O}_{\mathbb{A}^2_j}(\theta),\mathcal{O}_{\mathbb{A}^2_j}(\theta'))&=k[z_{3-j},z_3]/(z_{3-j}z_3)(\theta^{-1}\theta')^G,\\
H^{\text{odd}}\Hom(\mathcal{O}_{\mathbb{A}^2_j}(\theta),\mathcal{O}_{\mathbb{A}^2_j}(\theta'))&=0,\\
H^{\text{even}}\Hom(\mathcal{O}_{\mathbb{A}^2_j}(\theta),\mathcal{O}_{\mathbb{A}^2_{3-j}}(\theta'))&=0,\\
H^{\text{odd}}\Hom(\mathcal{O}_{\mathbb{A}^2_j}(\theta),\mathcal{O}_{\mathbb{A}^2_{3-j}}(\theta'))&=u_j\cdot k[z_3](\theta^{-1}\theta')^G.
\end{align*}
Note that ${\mathcal O}_{{\mathbb A}^2_3}(\theta)$ is isomorphic to $\mathrm{Cone}(u_1(\theta)^G)$, and thus $\{\mathcal{O}_{\mathbb{A}^2_3}(\theta)\}$ are generated by $\{\mathcal{O}_{\mathbb{A}^2_j}(\theta)\}_{j=1,2}$. 

The group $G$ acts on $u_j$ by $G$ via $g u_j=\rho_j^{-1}(g)u_j$, and $u_1 u_2=u_2 u_1=z_3$. Hence we have defined a functor $\coh(\mathcal{Y})\rightarrow\text{MF}(\mathcal{X},W)$. On the cohomology of the mapping complexes it takes $z_i$ to $z_i$, $u_i$ to $u_i$ and $v$ to $z_3$. It descends to a functor from $\mathbb{Z}/2$-folding of $\coh(\mathcal{Y})$, which is a quasi-isomorphism on the mapping complexes.
\end{proof}
\begin{rem}
	By breaking the symmetry in $\mathbb A^3$ we have effectively chosen preferred coordinates for the B-model computation. A different choice of such preferred coordinates corresponds to a different skeleton on the mirror curve.
\end{rem}
To compare $\coh_{\mathbb Z/2}(\mathcal Y)$ to the Fukaya category of the mirror curve, we use the following decomposition of the category.
\begin{prop}
The curve $\mathcal{Y}$ fits into a pushout diagram where $[*/G]$ embeds canonically as the origin
\begin{center}
\begin{tikzcd}
\text{[}*/G\text{]}\arrow{r}\arrow{d} & \text{[}\mathbb{A}^1_1/G\text{]}\arrow{d}\\
\text{[}\mathbb{A}^1_2/G\text{]}\arrow{r} & \mathcal{Y}.
\end{tikzcd}
\end{center}
After taking $\coh_{\mathbb{Z}/2}(-)$ this is turned into a pushout in $\text{DGCat}^{(2)}_{\text{small}}$
\begin{center}
\begin{tikzcd}
\coh_{\mathbb{Z}/2}(\text{[}*/G\text{]}\arrow{r}\arrow{d}) & \coh_{\mathbb{Z}/2}(\text{[}\mathbb{A}^1_1/G\text{]}\arrow{d})\\
\coh_{\mathbb{Z}/2}(\text{[}\mathbb{A}^1_2/G\text{]}\arrow{r}) & \coh_{\mathbb{Z}/2}(\mathcal{Y}).
\end{tikzcd}
\end{center}
\end{prop}
\begin{proof}
The first claim is obvious. The second claim is a consequence of \cite[Vol.2 Part II Chapter 8 Theorem A.1.2]{gaitsgory2017study}, as explained in \cite[Corollary 2.5]{nadler2016wrapped}.
\end{proof}
For the mirror curve $\mathcal{C}$, we construct an explicit skeleton from the equivariant structure and use it to compute the topological Fukaya category.

\begin{figure}
 \centering  \includegraphics[width=1\textwidth]{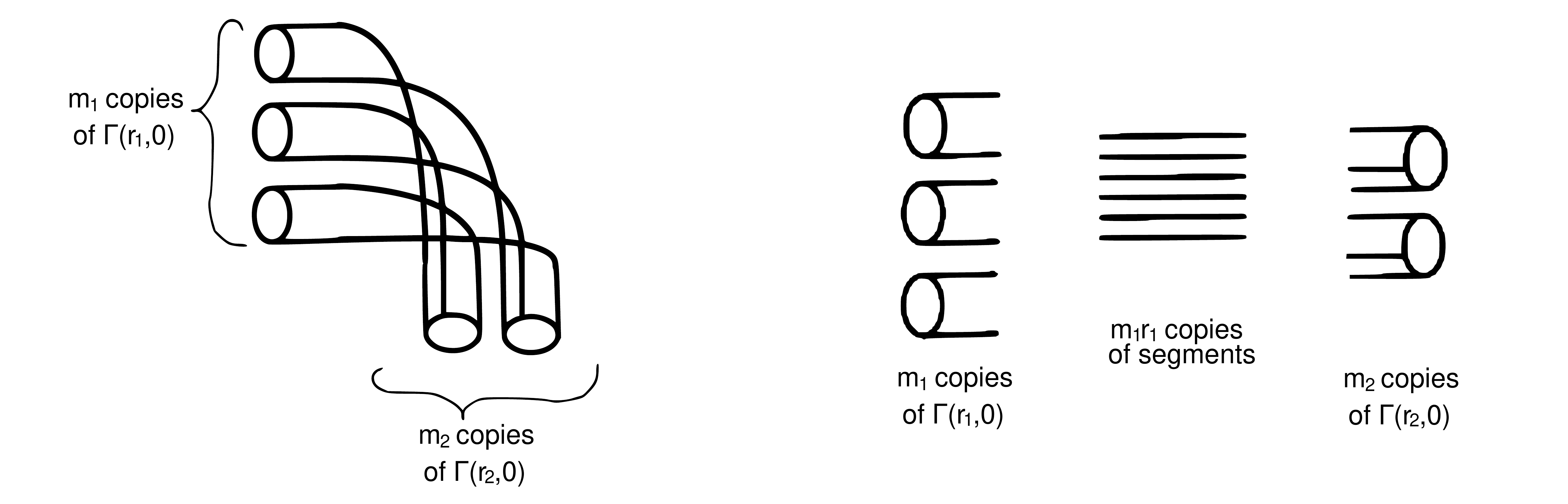} \caption{\label{skeleton}Skeleton and decomposition} 
\end{figure}
The the quotient of $\mathcal{C}$ under the $G^\vee$ action is a three-punctured sphere (Proposition \ref{quot}). We pick a dumbbell-shaped skeleton with two circles placed near the two punctures $X=0$ and $Y=0$, connected along the real line from the circle near $X=0, Y=-1$ to the circle near $X=-1, Y=0$ in the curve $1+X+Y=0$ (see Remark \ref{rmk:quotient}). The pullback of this skeleton is in Figure \ref{skeleton}.

We may decompose the skeleton $\Gamma$ as in Figure \ref{skeleton} and label them $\Gamma_i$ for $i=1,2,3$ from left to right. It is obvious that $\Gamma$ is covered by open subskeletons $\Gamma_1$ and $\Gamma_3$ with $\Gamma_1\cap\Gamma_3=\Gamma_2$. By cosheaf property of $\mathcal{F}^{top}(-)$, we get a decomposition of topological Fukaya category as a pushout
\begin{center}
\begin{tikzcd}
\mathcal{F}^{top}(\Gamma_2)\arrow{r}\arrow{d} & \mathcal{F}^{top}(\Gamma_1)\arrow{d}\\
\mathcal{F}^{top}(\Gamma_3)\arrow{r}    & \mathcal{F}^{top}(\Gamma).
\end{tikzcd}
\end{center}

\begin{prop}
There is a commutative diagram in $\text{DGCat}^{(2)}_{\text{small}}$ with vertical maps being equivalences
\begin{center}
\begin{tikzcd}
\mathcal{F}^{top}(\Gamma_1)\arrow{d}{\cong}&\mathcal{F}^{top}(\Gamma_2)\arrow{r}\arrow{l}\arrow{d}{\cong} & \mathcal{F}^{top}(\Gamma_3)\arrow{d}{\cong}\\
\coh_{\mathbb{Z}/2}(\text{[}\mathbb{A}^1_2/G\text{]})&\coh_{\mathbb{Z}/2}(\text{[}*/G\text{]}\arrow{l}\arrow{r}) & \coh_{\mathbb{Z}/2}(\text{[}\mathbb{A}^1_1/G\text{]}).
\end{tikzcd}
\end{center}
\end{prop}
\begin{proof} 
	We make use of the $G^\vee$-equivariant structure to produce the equivalence. Fix an arbitrary lift of dumbbell figure $\Gamma_D$ inside the skeleton $\Gamma$. We will use it to label the skeleton. For  $\Gamma_1=\cup_{m_1} \Gamma(r_1,0)$, there is a spoke picked out by the dumbbell figure and we label the closed interval on the central $S^1$ before it by $1\in G^\vee$. Then each closed interval is labeled by $\theta\in G^\vee$ via the $G^\vee$ action. The equivalence $\mathcal{F}^{top}(\Gamma_1)\cong\text{Perf}_{\mathbb{Z}/2}(\oplus_{m_1} Q(\Gamma(r_1,0)))$ in Example \ref{wheel} labels the quiver $\oplus_{m_1}Q(\Gamma(r_1,0))$ as follows: the vertices of $\oplus_{m_1} Q(\Gamma(r_1,0))$ are labeled by the closed intervals on the circles of $\Gamma_1$ and the arrows are labeled by the spokes. On the other hand, we can produce an equivalence from $\coh_{\mathbb{Z}/2}([\mathbb{A}^1_2/G])$ to $\text{Perf}_{\mathbb{Z}/2}(\oplus_{m_1} Q(\Gamma((r_1,0)))$, taking $M$ to a quiver representation whose value on the vertex $\theta$ is $M(\theta)^G$ and the arrows between them are multiplication by $z_1$. Note that $\rho_1$ acts on $\Gamma_1$ by rotation inside each wheel, so our construction makes sense. The case for $\Gamma_3$ is similar.

	For $\Gamma_2$ it is even easier: since $\Gamma_2$ is a disjoint union of intervals, we've picked out a segment labeled by $1$ from the dumbell $\Gamma_D$ and the others are labeled $\theta$. Therefore $\mathcal{F}^{top}(\Gamma_2)$ is just representation over discrete points and we can use the segments to label them. An equivalence between this representation and $\coh_{\mathbb{Z}/2}([*/G])$ is produced as before. The commutativity of the diagram is easy now: by equivariance, we only need to check locally at the base point labeled by $1$, this reduces to the computaion in the Example \ref{dumbbell}.
\end{proof}
\begin{rem}
	For each $\mathcal X_\sigma$ and a $2$-cone $\tau\in \sigma$, we may choose $\Gamma_D$ so that the labeling of the circles in $\Gamma_1$ matches the labeling of $\widetilde{\mathcal C_\tau}$; while the labeling of the circles in $\Gamma_2$ matches the labeling $\widetilde{\mathcal C_\tau'}$ where the $2$-cone $\tau'\in \sigma$ is counter-clockwisely adjacent to $\tau$ (c.f. Section \ref{sec:mirror-curve} ).
\end{rem}
\begin{cor} The matrix factorization category MF$(\mathcal{X}_\sigma,W)$ and the topological Fukaya category of its mirror curve $\mathcal{F}^{top}(\Gamma)$ are equivalent as $\mathbb{Z}/2$-dg categories.
\end{cor}

\begin{proof}
By above proposition, they are homotopy pushouts of isomorphic diagrams.
\end{proof}

\section{General Case}
We first recall some results of \cite{Pascaleff_2019}.
\subsection{Exceptional restriction functor}
We have defined restriction and corestriction functors for versions of topological Fukaya category. Now we turn to its Ind completion. It turns out that $\mathcal{F}^{top}_\infty$ enjoys even more functoriality.

Recall that given a open subgraph $U\subset \Gamma$, we have corestriction functor $C_U:\mathcal{F}^{top}(U)\rightarrow\mathcal{F}^{top}(\Gamma)$. Taking Ind completion gives a  functor $C_{\infty,U}:\mathcal{F}_\infty^{top}(U)\rightarrow\mathcal{F}_\infty^{top}(\Gamma)$ which we will also call corestriction. It is a cocontinuous functor between cocomplete categories, so we may take its right adjoint $R_\infty^U:\mathcal{F}_\infty^{top}(\Gamma)\rightarrow\mathcal{F}_\infty^{top}(U)$ and define it to be the restriction functor.

\begin{rem} (\cite[Remark 5.4]{Pascaleff_2019}) This restriction functor is not the Ind completion of the restriction functor for the compact topological Fukaya category. In fact, Ind$(\mathcal{F}_{top})$ is not necessarily the same as $\mathcal{F}^{top}_\infty$. 
\end{rem}
Since taking Ind completion preserves small colimits, we see that the diagram in Proposition \ref{glue} is taken to pushout and pullback, respectively
\begin{center}
\begin{tikzcd}
\mathcal{F}_\infty^{top}(U\cap V)\arrow{r}{C_{\infty}}\arrow{d}{C_{\infty}}&\mathcal{F}_\infty^{top}(U)\arrow{d}{C_{\infty}}\\
\mathcal{F}_\infty^{top}(V)\arrow{r}{C_{\infty}}&\mathcal{F}_\infty^{top}(U \cup V)
\end{tikzcd}
\begin{tikzcd}
\mathcal{F}_\infty^{top}(U\cap V)&\mathcal{F}_\infty^{top}(U)\arrow{l}{R_{\infty}}\\
\mathcal{F}_\infty^{top}(V)\arrow{u}{R_{\infty}}&\mathcal{F}_\infty^{top}(U\cup V)\arrow{l}{R_\infty}\arrow{u}{R_\infty}
\end{tikzcd}
\end{center}
by Ind completion and adjoint. Besides these restriction to open subgraph functors, we have an exceptional restriction to good closed subgraph functor defined by \cite{Dyckerhoff_2017}. By a good  open subgraph $U\subset\Gamma$ we mean an open subgraph whose complement having no vertices of valency $1$. By a good closed subgraph we mean a closed subgraph whose complement is good. We will only deal with restriction to good closed graphs in this paper and will sometimes just call it restriction to closed graphs.

Following \cite[Proposition 5.10]{Pascaleff_2019}, we denote the exceptional restriction functor  $S_\infty^Z:\mathcal{F}_\infty^{top}(\Gamma)\rightarrow\mathcal{F}_\infty^{top}(Z)$ for a good closed subgraph $Z\subset\Gamma$. We refer to loc. cit. for precise definition. The following is the main technical result of \cite{Pascaleff_2019} and will be used in the induction.

\begin{lemma}\label{close_compatible_with_open} (\cite[Proposition 5.14]{Pascaleff_2019}) Exceptional restriction is compatible with open restriction in the following sense. Suppose we have $V\subset U\subset X$ inclusions of ribbon graphs, with $V$ a good closed subgraph in $X$ and $U$ a open subgraph in $X$, then we have commutative diagram
\begin{center}
\begin{tikzcd}
\mathcal{F}_\infty^{top}(X)\arrow{dr}{S_\infty}\arrow{r}{R_\infty}&\mathcal{F}_\infty^{top}(U)\arrow{d}{S_{\infty}}\\
   &\mathcal{F}_\infty^{top}(V).
\end{tikzcd}
\end{center}
\end{lemma}

The topological Fukaya category of $\Gamma(p,q)$ can be computed as a homotopy pullback via exceptional restrictions. More generally, consider a ribbon graph $X$ and its two closed sugbraphs $Z_1,Z_2$ satisfying the following:
\begin{itemize}
	\item $Z_1\cup Z_2=X$;
	\item The underlying topological space of $Z_{1,2}=Z_1\cap Z_2$ is a disjoint union of circles;
	\item For each circle $C$ in $Z_{1,2}$, let $N_CX$ be the `neighborhood' of $C$ in $X$, obtained via adding to $C$ all the open edges of $X$ adjacent to $C$; we require the  inclusion $$N_CX\cap Z_1\subset N_CX \supset N_CX\cap Z_2$$ and to be isomorphic to $$\Gamma(n_1,0)\subset\Gamma(n_1,n_2)\supset\Gamma(0,n_2).$$
\end{itemize}

\begin{prop}\label{gluealongcircle}(\cite[Theorem 6.6]{Pascaleff_2019} ) 
With $Z_{1,2},Z_1,Z_2,X$ given as above,  we have a pullback diagram
\begin{center}
\begin{tikzcd}
\mathcal{F}_\infty^{top}(X)\arrow{r}{S_\infty}\arrow{d}{S_\infty}&\mathcal{F}_\infty^{top}(Z_1)\arrow{d}{S_\infty}\\
\mathcal{F}_\infty^{top}(Z_2)\arrow{r}{S_\infty}&\mathcal{F}_\infty^{top}(Z_{1,2}).
\end{tikzcd}
\end{center}
Moreover, in the local situation, i.e. $X=\Gamma(n_1,n_2)$, $Z_1=\Gamma(n_1,0)$, $Z_2=\Gamma(0,n_2)$, the diagram is isomorphic to the mirror pullback diagram in which all the maps are restriction:
\begin{center}
\begin{tikzcd}
\text{QCoh}_{\mathbb{Z}/2}(\mathbb{P}^1(p,q))\arrow{r}\arrow{d}&\text{QCoh}_{\mathbb{Z}/2}([\mathbb{A}^1/\mu_{n_1}])\arrow{d}\\
\text{QCoh}_{\mathbb{Z}/2}([\mathbb{A}^1/\mu_{n_2}])\arrow{r}&\text{QCoh}_{\mathbb{Z}/2}(\mathbb{G}_m)
\end{tikzcd}
\end{center}
via the comparison isomorphisms in Example \ref{wheel}.

\end{prop}

As explained in \label{restriction}(\cite[Definition 8.2]{Pascaleff_2019}), there is a restriction functor $R_p$ for each puncture $p$ on the punctured curve $\mathcal C$ defined as follows. Let $X$ be an arbitrary skeleton for $\mathcal{C}$. For each puncture $p$ of $\mathcal{C}$ such that $X$ does not approach $p$, we can choose another skeleton $X'$ such that it has a circle at $p$, which means that the component $\mathcal{C}-X'$ containing $p$ has an $S^1$ as boundary. This boundary can be viewed as a good closed subgraph of $X'$. We consider the composition
\begin{center}
\begin{tikzcd}
\mathcal{F}_\infty^{top}(X)\arrow{r}{\Psi}&\mathcal{F}_\infty^{top}(X')\arrow{r}{R_\infty}&\mathcal{F}_\infty^{top}(S^1).
\end{tikzcd}
\end{center}
We call this composition $R_p$ and it does not depend on the choice of the skeleton $X'$.

\subsection{Descent description} We give a descent description of the matrix factorization category. Recall that the combinatorial data for a Calabi-Yau 3-orbifold is a triangulated planar polygon $P_\Sigma$, whose faces, edges and vertices are indexed by the collection of cones $\Sigma(3),\Sigma(2)$ and $\Sigma(1)$ respectively. We may consider the dual graph $R_\Sigma$ of this triangulated polygon, so it has vertices indexed by $\text{Vertex}(R_\Sigma)=\Sigma(3)$ and edges indexed by $\text{Edge}(R_\Sigma)=\Sigma(2)$. We put the category $\text{MF}^\infty(\mathcal{X}_\sigma,W)\cong\text{MF}^\infty([\mathbb{A}^3/G],W=z_1z_2z_3)$ on each vertex $\sigma$. For an edge $\nu\in\Sigma(2)$,  we put a category $\text{MF}^{\infty}(\mathcal{X}_\nu,W)$ on it. There is a restriction functor from the vertices to the edges, given by the obivous Zariski open restriction functor of the matrix factorization category. We have assigned a $\mathbb{Z}/2$-dg category for each component of the graph and defined restriction functors, and we can actually think of it as a `constructible' sheaf of dg categories on the graph. We compute the value on the edges more explicitly as follows.

\begin{prop} \label{intersection} For the affine LG model $(\mathcal{X}_\sigma=[\mathbb{A}^3/G],W=z_1z_2z_3)$, there are three distinguished open charts in $\mathcal X_\sigma$ corresponding to three 2-dimensional subcones of $\sigma$. The cone $\nu\in\Sigma(2)$ generated by $\{(0,0,1),(0,m,1)\}$ gives an open chart $[\mathbb{G}_m\times\mathbb{A}^2/G]\subset[\mathbb{A}^3/G]$. We have
$$\text{MF}^{\infty}(\mathcal{X}_\nu,W)\cong\text{MF}^\infty([\mathbb{G}_m\times\mathbb{A}^2/G],W=z_1z_2z_3)\cong\text{QCoh}_{\mathbb{Z}/2}([\mathbb{G}_m/G]).$$
\end{prop}
\begin{proof} The proof parallels Proposition \ref{affinecomputation} except that we work on an open subscheme. The group $G$ acts on $\mathbb{G}_m$ via $\rho_1$. The category $\text{Coh}([\mathbb{G}_m/G])$ has generators $\{k[z_1,z_1^{-1}](\theta)\}$ with $\theta\in G^\vee$. One can first pullback along the projection $[\spec k[z_1,z_1^{-1},z_3]/G]\rightarrow[\spec k[z_1,z_1^{-1}]/G]$ then pushforward into $[\spec k[z_1,z_1^{-1},z_2,z_3]/(z_2z_3=0)/G]$. Projecting down to category of singularity this gives a functor
$$\text{Coh}([\spec k[z_1,z_1^{-1}]/G])\rightarrow D_{Sing}([\spec k[z_1,z_1^{-1},z_2,z_3]/(z_2z_3=0)/G]),$$ 
taking $\{k[z_1,z_1^{-1}](\theta)\}$ to $\{k[z_1,z_1^{-1},z_3](\theta)\}$ which is a set of generators for the category of singularity. Direct computation shows that after taking $\mathbb{Z}/2$-folding of the left-hand side this is a cohomologically isomorphism on mapping complex, hence is an equivalence. Note on the right hand side
$$D_{Sing}([\spec k[z_1,z_1^{-1},z_2,z_3]/(z_2z_3=0)/G])\cong\text{MF}([\spec k[z_1,z_1^{-1},z_2,z_3]/G],z_1z_2z_3).$$
After taking Ind completion we get the desired equivalence.
\end{proof}
\begin{rem} We have the following remarks in order.
	\begin{enumerate}
		\item
		The proof we adopt above makes explicit the commutativity of the diagram comparing $\text{MF}_\infty$ and $\mathcal{F}^{top}_\infty$ 
		\begin{center}
		\begin{tikzcd}
		\mathcal{F}_\infty^{top}(\Gamma)\arrow{r}{\cong}\arrow{d}{S_\infty=\bigoplus_p R_p}&\text{MF}^\infty([\mathbb{A}^3/G],W)\arrow{d}\\
		\mathcal{F}_\infty^{top}(\coprod_p S^1)\arrow{r}{\cong}&\text{MF}^\infty([\mathbb{G}_m\times\mathbb{A}^2/G],W)
		\end{tikzcd}
		\end{center} where $\Gamma$ is the skeleton for $\mathcal{C}_\sigma$ as in Figure \ref{skeleton}. In the equality $S_\infty=\bigoplus_p R_p$, we take direct sum over $p$ of all the punctures on the mirror curve corresponding to $\nu$ (there are $m$ copies of them). The same applies  $\coprod_p S^1$.
		\item We work with the preferred cone $\nu$ in our coordinates. A similar statement is true for the other two edges and can be computed similarly.
	\end{enumerate}
\end{rem}
Since we have a sheaf $B$ of dg categories on the graph $R_\Sigma$, we take its global section as the homotopy equalizer
$$B(R)=\text{Eq}[\prod_{\sigma\in\Sigma(3)} B(\sigma)\rightrightarrows \prod_{\nu\in\Sigma(2)} B(\nu)].$$
\begin{prop} We have the equivalence $B(R_\Sigma)\cong\text{MF}^\infty(\mathcal{X}_\Sigma,W)$.
\end{prop}
\begin{proof}
Note that $\{\mathcal{X}_\sigma\}$ is an open cover of $\mathcal{X}_\Sigma$. By the descent property of $\text{MF}^\infty(-,W)$, we can compute the global section by taking homotopy limit of C\v{e}ch complexes. For $(n>2)$-intersections, $\text{MF}^\infty(-,W)$ evaluates to be zero category because the zero locus $W^{-1}(0)$ has no singularity at all. By Remark \ref{cech} we see that $\text{MF}^\infty(-,W)$ is computed by the homotopy limit of the diagram as $B(R_\Sigma)$.
\end{proof}
We will eventually see that the topological Fukaya category of the mirror curve is computed by the same diagram.
\subsection{Proof of the main theorem}
Now we are ready to give the proof for  mirror symmetry in the general situation. 
\begin{prop}
Recall that by $\mathcal{F}^{top}_\infty(\mathcal{C})$, we mean to first pick a compact skeleton $X$ of $\mathcal{C}$ and take $\mathcal{F}^{top}_\infty(X)$. There is an equivalence  $\Psi:\mathcal{F}^{top}_\infty(\mathcal{C}_\Sigma)\cong B(R_\Sigma)$ such that the following diagram commutes:
\begin{center}
\begin{tikzcd}
\mathcal{F}_\infty^{top}(\mathcal{C}_\Sigma)\arrow{r}{\Psi,\cong}\arrow{d}{\bigoplus_{p\in\text{Pun}(\nu)} R_p}&B(R_\Sigma)\arrow{d}\\
\mathcal{F}_\infty^{top}(\coprod_{\text{Pun}(\nu)} S^1)\arrow{r}{\Psi,\cong}&B(\nu).
\end{tikzcd}
\end{center}
Each $\nu\in\Sigma(2)$ corresponding to a segment lying on the boundary of the polygon determines a collection of punctures $\text{Pun}(\nu)$ in $\mathcal{C}_\Sigma$ and we have a well-defined restriction functor to them by Proposition \ref{restriction}, while on the other hand $B(R_\Sigma)$ has a restriction functor to its evaluation on this edge $B(\nu)$, which corresponds to the Zariski open restriction functor of the matrix factorization.
\end{prop}
\begin{proof} 
	
	We do induction on the cardinality of $\Sigma(3)$. As explained in \cite[Section 7.2]{Pascaleff_2019}, each triangulated polygon can be glued step by step from triangles. For each gluing step, one only needs to glue in a new triangle along one or two of its edges. When $\#\Sigma(3)=1$, we are back in the affine situation. The equivalence of restriction functors is verified in Proposition \ref{intersection}. 
	
	Assume the case for $\#\Sigma(3)=n$ is true. Given a fan $\Sigma$ with $\#\Sigma(3)=n+1$, we can decompose it as $\Sigma(3)=\Sigma'\cup \sigma$ where $\sigma$ is $3$-dimensional cone and $\Sigma'(3)=n$. For the corresponding polygon, this decomposition describes $P_\Sigma$ as obtained from $P_{\Sigma'}$ by gluing in a new triangle $P_\sigma$ along one or two of its edges. On the mirror curve side, this corresponds to attaching a mirror curve $\mathcal{C}_\sigma$ to $\mathcal{C}_{\Sigma'}$ along the punctures labeled by the gluing edges of the triangle. For simplicity, let's assume that we are gluing along one edge $\nu$ of the triangle $\sigma$. The case of gluing two edges is done similarly. 
	
	By \cite[Lemma 7.5]{Pascaleff_2019}, we can pick a skeleton $X$ for $\mathcal{C}_{\Sigma'}$ that has disjoint circles at gluing punctures. Similarly, we pick a skeleton $Y$ for $\mathcal{C}_\sigma$ that has disjoint circles at gluing punctures. For a gluing puncture, the skeleton $X$ and $Y$ each have a unique gluing circle. We can identify these $S^1$ circles and obtain the glued skeleton for $\mathcal{C}_\Sigma$. This process is described by the   pushout of ribbon graphs
	\begin{center}
		\begin{tikzcd}
			\coprod_{\text{Pun}(\nu)} S^1\arrow{r}\arrow{d}&X\arrow{d}\\
			Y\arrow{r}&X \coprod_{\coprod_{\text{Pun}(\nu)} S^1} Y.
		\end{tikzcd}
	\end{center}
	The resulting $X \coprod_{\coprod_{\text{Pun}(\nu)} S^1} Y$ is a skeleton for $\mathcal{C}_\Sigma$. As discussed in Proposition \ref{sec:general-mirror-curve}, the cyclic group $G^\vee_\nu\cong \mu_n$ acts on the index set for the circles, where $n+1=\#\text{Pun}(\nu)+1$ is the number of the lattice points on the gluing edge $\nu$. From our explicit description of the action, one sees that the gluing description of the mirror curve is compatible with the action from both sides. This is important for assembling the restriction functors $R_p$ to $\bigoplus R_p$ and compare with the mirror side.
	
	The decomposition of $X \coprod_{\coprod_{\text{Pun}(\nu)} S^1} Y$ by two closed subgraphs $\{X,Y\}$ satisfies the assumption of Proposition \ref{gluealongcircle}. We conclude that the topological Fukaya category is the following homotopy pullback
	\begin{center}
		\begin{tikzcd}
			\mathcal{F}_\infty^{top}(X \coprod_{\coprod_{\text{Pun}(\nu)} S^1} Y)\arrow{r}{S_\infty}\arrow{d}{S_\infty}&\mathcal{F}_\infty^{top}(X)\arrow{d}{S_\infty}\\
			\mathcal{F}_\infty^{top}(Y)\arrow{r}{S_\infty}&\mathcal{F}_\infty^{top}(\coprod_{\text{Pun}(\nu)} S^1).
		\end{tikzcd}
	\end{center}
	By the induction assumption, we have the following commutative diagram:
	\begin{center}
		\begin{tikzcd}
			\mathcal{F}_\infty^{top}(Y)\arrow{d}{\Psi,\cong}\arrow{r}{S_\infty}&\mathcal{F}_\infty^{top}(\coprod_{\text{Pun}(\nu)} S^1)\arrow{d}{\Psi,\cong}&\mathcal{F}_\infty^{top}(X)\arrow{l}{S_\infty}\arrow{d}{\Psi,\cong}\\
			B(R_\Sigma)\arrow{r}&B(\nu)&B(R_\sigma)\arrow{l}.
		\end{tikzcd}
	\end{center}
	Taking homotopy pullback of the lower row gives a diagram
	\begin{center}
		\begin{tikzcd}
			B(R_\Sigma)\arrow{r}\arrow{d}&B(R_\sigma)\arrow{d}\\
			B(R_{\Sigma'})\arrow{r}&B(\nu).
		\end{tikzcd}
	\end{center}
		Comparing these two squares, we get the desired equivalence$$\Psi:\mathcal{F}^{top}_\infty(X \coprod_{\coprod_{\text{Pun}(\nu)} S^1} Y)\cong B(R_\Sigma).$$

	The more difficult part is to analyze the functors of the restriction to punctures. We claim that these restriction functors factor through the exceptional restrictions to $X$ or $Y$, i.e., for each single puncture $p$ coming from $\mathcal{C}_{\Sigma'}$ the following diagram commutes, where the map on the right (abusively named $R_p$) is the restriction functor `defined inside $X$':
	\begin{center}
		\begin{tikzcd}
			\mathcal{F}_\infty^{top}(X \coprod_{\coprod_{\text{Pun}(\nu)} S^1} Y)\arrow{r}{S_\infty}\arrow{rd}{R_p} & \mathcal{F}_\infty^{top}(X)\arrow{d}{R_p}\\
			&\mathcal{F}_\infty^{top}(S^1).
		\end{tikzcd}
	\end{center}
	Collecting all restriction functors associated to the punctures on an edge $\nu$ from $\Sigma'$ gives a commutative diagram by the induction hypothesis
	\begin{center}
		\begin{tikzcd}
			\mathcal{F}_\infty^{top}(X \coprod_{\coprod_{\text{Pun}(\nu)} S^1} Y)\arrow{r}{S_\infty}\arrow{d}{\Psi,\cong}&\mathcal{F}_\infty^{top}(X)\arrow{r}{\bigoplus_{\text{Pun}(\nu)} R_p}\arrow{d}{\Psi,\cong}&\mathcal{F}_\infty^{top}(\coprod_{\text{Pun}(\nu)} S^1)\arrow{d}{\Psi,\cong}\\
			B(R_{\Sigma})\arrow{r}&B(R_{\Sigma'})\arrow{r}&B(\nu),
		\end{tikzcd}
	\end{center}
	and this verifies the required commutativity. Now it remains to check the claim. This follows exactly as in \cite[Theorem 8.3]{Pascaleff_2019}. First let's recall the defintion of the restriction functor $R_p$: we need to modify the skeleton of $\mathcal{C}_\Sigma$ to obtain a circle at $p$ then apply exceptional restrction functor. Suppose we are looking at a puncture $p$ from $\Sigma'$. We only need to modify the skeleton $X \coprod_{\coprod_{\text{Pun}(\nu)} S^1} Y$ inside a neighbourhood of  $X$  to produce a puncture at $p$. As the exceptional restriction is compatible with open restriction in Lemma \ref{close_compatible_with_open}, it suffices to understand the restriction functor from an open neighbourhood $U$ of $X$ to $p$. Note that the exceptional restriction $S_\infty$ from $U$ to $X$ is in fact modding out the image of corestriction functors from some external edges, and it doesn't matter whether we do this before or after modifying $U$. So we conclude that the restriction functor $R_p:\mathcal{F}^{top}_\infty(X \coprod_{\coprod_{\text{Pun}(\nu)} S^1} Y)\rightarrow\mathcal{F}^{top}_\infty(S^1)$ actually factors through exceptional restriction to $X$ and the above argument applies.

\end{proof}
\begin{cor} \label{mainresult}We have the mirror symmetry equivalence $\mathcal{F}^{top}_\infty(\mathcal{C}_\Sigma)\cong \text{MF}^\infty(\mathcal{X}_\Sigma)$.
\end{cor}
\printbibliography

@misc{nadler2021sheaf,
      title={Sheaf quantization in Weinstein symplectic manifolds}, 
      author={David Nadler and Vivek Shende},
      year={2021},
      eprint={2007.10154},
      archivePrefix={arXiv},
      primaryClass={math.SG}
}

@article{Haiden_2017,
    doi = {10.1007/s10240-017-0095-y},
  
    url = {https://doi.org/10.1007%2Fs10240-017-0095-y},
  
    year = 2017,
    month = {nov},
  
    publisher = {Springer Science and Business Media {LLC}
},
  
    volume = {126},
  
    number = {1},
  
    pages = {247--318},
  
    author = {F. Haiden and L. Katzarkov and M. Kontsevich},
  
    title = {Flat surfaces and stability structures},
  
    journal = {Publications math{\'{e}}matiques de l{\textquotesingle}{IH}{\'{E}}S}
}

@article {EOProof,
    AUTHOR = {Eynard, B. and Orantin, N.},
     TITLE = {Computation of open {G}romov-{W}itten invariants for toric
              {C}alabi-{Y}au 3-folds by topological recursion, a proof of
              the {BKMP} conjecture},
   JOURNAL = {Comm. Math. Phys.},
  FJOURNAL = {Communications in Mathematical Physics},
    VOLUME = {337},
      YEAR = {2015},
    NUMBER = {2},
     PAGES = {483--567},
      ISSN = {0010-3616},
   MRCLASS = {14N35 (14J33 53D45 81R05 81R10)},
  MRNUMBER = {3339157},
MRREVIEWER = {Sergiy Koshkin},
       DOI = {10.1007/s00220-015-2361-5},
       URL = {https://doi.org/10.1007/s00220-015-2361-5},
}

@article {BKMP,
    AUTHOR = {Bouchard, Vincent and Klemm, Albrecht and Mari\~{n}o, Marcos and
              Pasquetti, Sara},
     TITLE = {Remodeling the {B}-model},
   JOURNAL = {Comm. Math. Phys.},
  FJOURNAL = {Communications in Mathematical Physics},
    VOLUME = {287},
      YEAR = {2009},
    NUMBER = {1},
     PAGES = {117--178},
      ISSN = {0010-3616},
   MRCLASS = {81T45 (32G81 32Q25 81T30)},
  MRNUMBER = {2480744},
MRREVIEWER = {Johannes Walcher},
       DOI = {10.1007/s00220-008-0620-4},
       URL = {https://doi.org/10.1007/s00220-008-0620-4},
}

@misc {AganagicVafa,
      title={Mirror Symmetry, D-Branes and Counting Holomorphic Discs}, 
      author={Aganagic, Mina and Vafa, Cumrun},
      year={2002},
      eprint={hep-th/0012041},
      archivePrefix={arXiv},
}

@article {AKV,
    AUTHOR = {Aganagic, Mina and Klemm, Albrecht and Vafa, Cumrun},
     TITLE = {Disk instantons, mirror symmetry and the duality web},
   JOURNAL = {Z. Naturforsch. A},
  FJOURNAL = {Zeitschrift f\"{u}r Naturforschung. A. Journal of Physical
              Sciences},
    VOLUME = {57},
      YEAR = {2002},
    NUMBER = {1-2},
     PAGES = {1--28},
      ISSN = {0932-0784},
   MRCLASS = {81T30 (14J81 32Q25 81T60)},
  MRNUMBER = {1906661},
MRREVIEWER = {Chiu-Chu Melissa Liu},
       DOI = {10.1515/zna-2002-9-1001},
       URL = {https://doi.org/10.1515/zna-2002-9-1001},
}

@misc{KontsevichProposal,
  title  = {Symplectic geometry of homological algebra},
  author = {Kontsevich, Maxim},
  year   = {2009},
  url    = {http://www.ihes.fr/~maxim/TEXTS/Symplectic_AT2009.pdf},
}

@inproceedings {KontsevichICM,
    AUTHOR = {Kontsevich, Maxim},
     TITLE = {Homological algebra of mirror symmetry},
 BOOKTITLE = {Proceedings of the {I}nternational {C}ongress of
              {M}athematicians, {V}ol. 1, 2 ({Z}\"{u}rich, 1994)},
     PAGES = {120--139},
 PUBLISHER = {Birkh\"{a}user, Basel},
      YEAR = {1995},
   MRCLASS = {32J25 (14D07 14J32 18E30 32G05)},
  MRNUMBER = {1403918},
MRREVIEWER = {Claire Voisin},
}

@article {fang2019remodeling,
author = {Fang, Bohan and Liu, Chiu-Chu Melissa and Zong, Zhengyu},
title = {On the remodeling conjecture for toric {C}alabi-{Y}au
              3-orbifolds},
journal = {J. Amer. Math. Soc.},
fjournal = {Journal of the American Mathematical Society},
volume = {33},
year = {2020},
number = {1},
pages = {135--222},
issn = {0894-0347},
mrclass = {14N35 (14J33)},
mrnumber = {4066474},
doi = {10.1090/jams/934},
url = {https://doi.org/10.1090/jams/934},
}

@book {FOOO,
    AUTHOR = {Fukaya, Kenji and Oh, Yong-Geun and Ohta, Hiroshi and Ono,
              Kaoru},
     TITLE = {Lagrangian intersection {F}loer theory: anomaly and
              obstruction. {P}art {I}},
    SERIES = {AMS/IP Studies in Advanced Mathematics},
    VOLUME = {46},
 PUBLISHER = {American Mathematical Society, Providence, RI; International
              Press, Somerville, MA},
      YEAR = {2009},
     PAGES = {xii+396},
      ISBN = {978-0-8218-4836-4},
   MRCLASS = {53D40 (53D12 53D37)},
  MRNUMBER = {2553465},
MRREVIEWER = {Michael J. Usher},
       DOI = {10.1090/crmp/049/07},
       URL = {https://doi.org/10.1090/crmp/049/07},
}

@article{Dyckerhoff_2017,
    AUTHOR = {Dyckerhoff, Tobias},
     TITLE = {{$\Bbb{A}^1$}-homotopy invariants of topological {F}ukaya
              categories of surfaces},
   JOURNAL = {Compos. Math.},
  FJOURNAL = {Compositio Mathematica},
    VOLUME = {153},
      YEAR = {2017},
    NUMBER = {8},
     PAGES = {1673--1705},
      ISSN = {0010-437X},
   MRCLASS = {18G55 (14F42 53D37)},
  MRNUMBER = {3705272},
MRREVIEWER = {Fabio Ferrari Ruffino},
       DOI = {10.1112/S0010437X17007205},
       URL = {https://doi.org/10.1112/S0010437X17007205},
}

@article{Borisov_2004,
   title={{The orbifold Chow ring of toric Deligne-Mumford stacks}},
   volume={18},
   ISSN={1088-6834},
   url={http://dx.doi.org/10.1090/S0894-0347-04-00471-0},
   DOI={10.1090/s0894-0347-04-00471-0},
   number={1},
   journal={Journal of the American Mathematical Society},
   publisher={American Mathematical Society (AMS)},
   author={Borisov, Lev A. and Chen, Linda and Smith, Gregory G.},
   year={2004},
   month={Nov},
   pages={193–215}
}

@article{toen2007homotopy,
  title={{The homotopy theory of dg-categories and derived Morita theory}},
  author={To{\"e}n, Bertrand},
  journal={Inventiones mathematicae},
  volume={167},
  number={3},
  pages={615--667},
  year={2007},
  publisher={Springer}
}

@article{benzvi2010integral,
    AUTHOR = {Ben-Zvi, David and Francis, John and Nadler, David},
     TITLE = {Integral transforms and {D}rinfeld centers in derived
              algebraic geometry},
   JOURNAL = {J. Amer. Math. Soc.},
  FJOURNAL = {Journal of the American Mathematical Society},
    VOLUME = {23},
      YEAR = {2010},
    NUMBER = {4},
     PAGES = {909--966},
      ISSN = {0894-0347},
   MRCLASS = {14D23 (14F05 18D10 18E30)},
  MRNUMBER = {2669705},
MRREVIEWER = {Andrei D. Halanay},
       DOI = {10.1090/S0894-0347-10-00669-7},
       URL = {https://doi.org/10.1090/S0894-0347-10-00669-7},
}

@book{gaitsgory2017study,
    AUTHOR = {Gaitsgory, Dennis and Rozenblyum, Nick},
     TITLE = {A study in derived algebraic geometry. {V}ol. {I}.
              {C}orrespondences and duality},
    SERIES = {Mathematical Surveys and Monographs},
    VOLUME = {221},
 PUBLISHER = {American Mathematical Society, Providence, RI},
      YEAR = {2017},
     PAGES = {xl+533pp},
      ISBN = {978-1-4704-3569-1},
   MRCLASS = {14F05 (18D05 18G55)},
  MRNUMBER = {3701352},
MRREVIEWER = {Adrian Langer},
       DOI = {10.1090/surv/221.1},
       URL = {https://doi.org/10.1090/surv/221.1},
}

@article{abouzaid2014homological,
    AUTHOR = {Abouzaid, Mohammed and Auroux, Denis and Efimov, Alexander I.
              and Katzarkov, Ludmil and Orlov, Dmitri},
     TITLE = {{Homological mirror symmetry for punctured spheres}},
   journal = {J. Amer. Math. Soc.},
  FJOURNAL = {Journal of the American Mathematical Society},
    VOLUME = {26},
      YEAR = {2013},
    NUMBER = {4},
     PAGES = {1051--1083},
      ISSN = {0894-0347},
   MRCLASS = {53D37 (14F05 14J33 18E30 53D12)},
  MRNUMBER = {3073884},
MRREVIEWER = {Artan Sheshmani},
       DOI = {10.1090/S0894-0347-2013-00770-5},
       URL = {https://doi.org/10.1090/S0894-0347-2013-00770-5},
}

@article{Pascaleff_2019,
   title={{Topological Fukaya category and mirror symmetry for punctured surfaces}},
   volume={155},
   ISSN={1570-5846},
   url={http://dx.doi.org/10.1112/S0010437X19007073},
   DOI={10.1112/s0010437x19007073},
   number={3},
   journal={Compositio Mathematica},
   publisher={Wiley},
   author={Pascaleff, James and Sibilla, Nicolò},
   year={2019},
   month={Mar},
   pages={599–644}
}

@article{fang2019genus,
    AUTHOR = {Fang, Bohan and Liu, Chiu-Chu Melissa and Zong, Zhengyu},
     TITLE = {All-genus open-closed mirror symmetry for affine toric
              {C}alabi-{Y}au 3-orbifolds},
   JOURNAL = {Algebr. Geom.},
  FJOURNAL = {Algebraic Geometry},
    VOLUME = {7},
      YEAR = {2020},
    NUMBER = {2},
     PAGES = {192--239},
      ISSN = {2313-1691},
   MRCLASS = {14N35 (14J33)},
  MRNUMBER = {4061327},
       DOI = {10.14231/ag-2020-007},
       URL = {https://doi.org/10.14231/ag-2020-007},
}

@book{preygel2011thomsebastiani,
    AUTHOR = {Preygel, Anatoly},
     TITLE = {Thom-{S}ebastiani and {D}uality for {M}atrix {F}actorizations,
              and {R}esults on the {H}igher {S}tructures of the {H}ochschild
              {I}nvariants},
      NOTE = {Thesis (Ph.D.)--Massachusetts Institute of Technology},
 PUBLISHER = {ProQuest LLC, Ann Arbor, MI},
      YEAR = {2012},
     PAGES = {(no paging)},
   MRCLASS = {Thesis},
  MRNUMBER = {3121870},
       URL =
              {http://gateway.proquest.com/openurl?url_ver=Z39.88-2004&rft_val_fmt=info:ofi/fmt:kev:mtx:dissertation&res_dat=xri:pqm&rft_dat=xri:pqdiss:0828804},
}

@article{Orlov_2011,
   title={{Matrix factorizations for nonaffine LG–models}},
   volume={353},
   ISSN={1432-1807},
   url={http://dx.doi.org/10.1007/s00208-011-0676-x},
   DOI={10.1007/s00208-011-0676-x},
   number={1},
   journal={Mathematische Annalen},
   publisher={Springer Science and Business Media LLC},
   author={Orlov, Dmitri},
   year={2011},
   month={May},
   pages={95–108}
}

@article{nadler2016wrapped,
      title={{Wrapped microlocal sheaves on pairs of pants}}, 
      author={David Nadler},
      year={2016},
      eprint={1604.00114},
      archivePrefix={arXiv},
      primaryClass={math.SG}
}

@article{Or04,
author = {Orlov, Dmitri},
year = {2004},
month = {07},
pages = {227–248},
title = {{Triangulated categories of singularities and D-branes in Landau-Ginzburg models}},
volume = {246},
journal = {Proceedings of the Steklov Institute of Mathematics}
}

@article{dyckerhoff2013triangulated,
    AUTHOR = {Dyckerhoff, Tobias and Kapranov, Mikhail},
     TITLE = {Triangulated surfaces in triangulated categories},
   JOURNAL = {J. Eur. Math. Soc. (JEMS)},
  FJOURNAL = {Journal of the European Mathematical Society (JEMS)},
    VOLUME = {20},
      YEAR = {2018},
    NUMBER = {6},
     PAGES = {1473--1524},
      ISSN = {1435-9855},
   MRCLASS = {18E30 (57Q15)},
  MRNUMBER = {3801819},
MRREVIEWER = {Julia Ramos Gonz\'{a}lez},
       DOI = {10.4171/JEMS/791},
       URL = {https://doi.org/10.4171/JEMS/791},
}

@book {SeidelBook,
    AUTHOR = {Seidel, Paul},
     TITLE = {Fukaya categories and {P}icard-{L}efschetz theory},
    SERIES = {Zurich Lectures in Advanced Mathematics},
 PUBLISHER = {European Mathematical Society (EMS), Z\"{u}rich},
      YEAR = {2008},
     PAGES = {viii+326},
      ISBN = {978-3-03719-063-0},
   MRCLASS = {53D40 (16E45 32Q65 53D12)},
  MRNUMBER = {2441780},
MRREVIEWER = {Timothy Perutz},
       DOI = {10.4171/063},
       URL = {https://doi.org/10.4171/063},
}

@article{sibilla2011ribbon,
    AUTHOR = {Sibilla, Nicol\`o and Treumann, David and Zaslow, Eric},
     TITLE = {Ribbon graphs and mirror symmetry},
   JOURNAL = {Selecta Math. (N.S.)},
  FJOURNAL = {Selecta Mathematica. New Series},
    VOLUME = {20},
      YEAR = {2014},
    NUMBER = {4},
     PAGES = {979--1002},
      ISSN = {1022-1824},
   MRCLASS = {14J33 (14F05 32S60 53D37)},
  MRNUMBER = {3273628},
MRREVIEWER = {Nick Sheridan},
       DOI = {10.1007/s00029-014-0149-7},
       URL = {https://doi.org/10.1007/s00029-014-0149-7},
}

@article {ccit2020,
    AUTHOR = {Coates, Tom and Corti, Alessio and Iritani, Hiroshi and Tseng,
              Hsian-Hua},
     TITLE = {Hodge-theoretic mirror symmetry for toric stacks},
   JOURNAL = {J. Differential Geom.},
  FJOURNAL = {Journal of Differential Geometry},
    VOLUME = {114},
      YEAR = {2020},
    NUMBER = {1},
     PAGES = {41--115},
      ISSN = {0022-040X},
   MRCLASS = {14J33 (14F10 14M25 53D37)},
  MRNUMBER = {4047552},
       DOI = {10.4310/jdg/1577502022},
       URL = {https://doi.org/10.4310/jdg/1577502022},
}

@article{Kuwagaki_2020,
   title={The nonequivariant coherent-constructible correspondence for toric stacks},
   volume={169},
   ISSN={0012-7094},
   url={http://dx.doi.org/10.1215/00127094-2020-0011},
   DOI={10.1215/00127094-2020-0011},
   number={11},
   journal={Duke Mathematical Journal},
   publisher={Duke University Press},
   author={Kuwagaki, Tatsuki},
   year={2020},
   month={Aug}
}

@misc{lee2016homological,
      title={Homological mirror symmetry for open Riemann surfaces from pair-of-pants decompositions}, 
      author={Heather Lee},
      year={2016},
      eprint={1608.04473},
      archivePrefix={arXiv},
      primaryClass={math.SG}
}

@misc{ganatra2020microlocal,
      title={Microlocal Morse theory of wrapped Fukaya categories}, 
      author={Sheel Ganatra and John Pardon and Vivek Shende},
      year={2020},
      eprint={1809.08807},
      archivePrefix={arXiv},
      primaryClass={math.SG}
}

@article{Ganatra_2019,
   title={Covariantly functorial wrapped Floer theory on Liouville sectors},
   volume={131},
   ISSN={1618-1913},
   url={http://dx.doi.org/10.1007/s10240-019-00112-x},
   DOI={10.1007/s10240-019-00112-x},
   number={1},
   journal={Publications mathématiques de l’IHÉS},
   publisher={Springer Science and Business Media LLC},
   author={Ganatra, Sheel and Pardon, John and Shende, Vivek},
   year={2019},
   month={Aug},
   pages={73–200}
}

@misc{ganatra2019sectorial,
      title={Sectorial descent for wrapped Fukaya categories}, 
      author={Sheel Ganatra and John Pardon and Vivek Shende},
      year={2019},
      eprint={1809.03427},
      archivePrefix={arXiv},
      primaryClass={math.SG}
}
\end{document}